\documentclass[12pt]{amsart}

\usepackage{amssymb,amscd,amsthm, verbatim,amsmath,color,fancyhdr, mathrsfs}
\usepackage{graphicx}
\usepackage{turnstile}

\usepackage[letterpaper, left=2.5cm, right=2.5cm, top=2.5cm,
bottom=2.5cm,dvips]{geometry}

\usepackage{amsthm, amssymb, amsmath,amsfonts, amsthm, mathtools, enumitem, stackrel}
\usepackage{color}
\usepackage{tikz, tikz-cd}
\usepackage{comment}
\usepackage{adjustbox}
\usepackage{placeins}
\usepackage[foot]{amsaddr}
\usepackage{hyperref}
\hypersetup{citebordercolor=blue}
\newcommand{\CA}{\mathcal{A}}
\newcommand{\CB}{\mathcal{B}}

\newcommand{\FF}{\mathbb{F}}

\newcommand{\Ext}{\mathrm{Ext}}
\newcommand{\Hom}{\mathrm{Hom}}
\newcommand{\im}{\mathrm{im}}
\newcommand{\Mod}{\mathrm{-Mod}}
\newcommand{\Stab}{\mathrm{Stab}}

\newcommand{\emd}{\text{---}}
\newcommand{\mm}{/\!/}
\newcommand{\wh}{\widehat}
\newcommand{\wt}{\widetilde}

\newtheorem{thm}{Theorem}[section]
\newtheorem{cor}[thm]{Corollary}

\newtheorem{lemma}[thm]{Lemma}
\newtheorem{prop}[thm]{Proposition}
\newtheorem*{classthm}{Theorem~\ref{thmclassification}}
\newtheorem*{d1prop}{Theorem~\ref{thmd1}}
\newtheorem*{liftcor}{Corollary~\ref{corliftingcrit}}
\newtheorem*{e1}{Theorem~\ref{thmE1}}

\theoremstyle{definition}
\newtheorem{defn}[thm]{Definition}
\newtheorem{rmk}[thm]{Remark}
\newtheorem{exmp}[thm]{Example}

\newtheorem{notation}[thm]{Notation}
\newtheorem{note}[thm]{Note}
\newtheorem{conj}[thm]{Conjecture}

\tikzset{
    turn/.style={anchor=south, rotate=270, inner sep=.5mm}
}

\def\sqtwoL (#1,#2,#3){
  \draw[#3, thick] (#1,#2) .. controls (#1-1,#2+1) .. (#1,#2+2);
}

\def\sqtwoR (#1,#2,#3){
  \draw[#3, thick] (#1,#2) .. controls (#1+1,#2+1) .. (#1,#2+2);
}

\def \sqtwoCR (#1,#2,#3){
   \draw[#3, thick] (#1,#2) .. controls (#1+1,#2+.5) and (#1+1.5,#2+2) .. (#1+2,#2+2);
}

\def \sqtwoCL (#1,#2,#3){
   \draw[#3, thick] (#1,#2) .. controls (#1-1,#2+.5) and (#1-1.5,#2+2)  .. (#1-2,#2+2);
}

\def \sqone (#1,#2,#3){
  \draw[#3, thick] (#1,#2) -- (#1,#2+1);
}

\def \sqfourL (#1,#2,#3){
\draw[#3, thick] (#1,#2) -- (#1-.5, #2)--(#1-.5, #2+4)--(#1, #2+4);
}

\def\Aone (#1,#2){
\fill (#1, #2) circle (3pt);
\fill (#1, #2+1) circle (3pt);
\fill (#1, #2+2) circle (3pt);
\fill (#1, #2+3) circle (3pt);
\fill (#1+2, #2+3) circle (3pt);
\fill (#1+2, #2+4) circle (3pt);
\fill (#1+2, #2+5) circle (3pt);
\fill (#1+2, #2+6) circle (3pt);
\draw (#1, #2) -- (#1, #2+1);
\draw (#1, #2+2) -- (#1, #2+3);
\draw (#1+2, #2+3) -- (#1 + 2, #2+4);
\draw (#1+2, #2+5) -- (#1+2, #2+6);
\draw (#1, #2) .. controls (#1-1, #2+1) .. (#1, #2+2);
\draw (#1+2, #2+4) .. controls (#1+3, #2+5) .. (#1+2, #2+6);
\draw (#1, #2+1) .. controls (#1+1, #2+1.5) and  (#1+1.5 ,#2+3) .. (#1+2,#2+3);
\draw (#1, #2+2) .. controls (#1+1, #2+2.5) and (#1+1.5, #2+4) .. (#1+2, #2+4);
\draw (#1, #2+3) .. controls (#1+1, #2+3.5) and (#1+1.5, #2+5) .. (#1+2, #2+5);
}

\def\Aonecolor (#1,#2, #3){
\fill[#3] (#1, #2) circle (3pt);
\fill[#3]  (#1, #2+1) circle (3pt);
\fill[#3]  (#1, #2+2) circle (3pt);
\fill[#3]  (#1, #2+3) circle (3pt);
\fill[#3]  (#1+2, #2+3) circle (3pt);
\fill[#3]  (#1+2, #2+4) circle (3pt);
\fill[#3]  (#1+2, #2+5) circle (3pt);
\fill[#3]  (#1+2, #2+6) circle (3pt);
\draw[#3]  (#1, #2) -- (#1, #2+1);
\draw[#3]  (#1, #2+2) -- (#1, #2+3);
\draw[#3]  (#1+2, #2+3) -- (#1 + 2, #2+4);
\draw[#3]  (#1+2, #2+5) -- (#1+2, #2+6);
\draw[#3]  (#1, #2) .. controls (#1-1, #2+1) .. (#1, #2+2);
\draw[#3]  (#1+2, #2+4) .. controls (#1+3, #2+5) .. (#1+2, #2+6);
\draw[#3]  (#1, #2+1) .. controls (#1+1, #2+1.5) and  (#1+1.5 ,#2+3) .. (#1+2,#2+3);
\draw[#3]  (#1, #2+2) .. controls (#1+1, #2+2.5) and (#1+1.5, #2+4) .. (#1+2, #2+4);
\draw[#3]  (#1, #2+3) .. controls (#1+1, #2+3.5) and (#1+1.5, #2+5) .. (#1+2, #2+5);
}

\def\rectangle (#1,#2,#3){   \draw[#3] (#1-0.15,#2-0.15) rectangle (#1+0.15,#2+0.15)}

\def\SgL (#1,#2){
\fill (#1, #2) circle (3pt);
\fill (#1, #2+2) circle (3pt);
\fill (#1, #2+3) circle (3pt);
\fill (#1, #2+5) circle (3pt);
\draw[thick] (#1, #2+2)--(#1, #2+3);
\draw[thick] (#1, #2) .. controls (#1-1, #2+1) .. (#1, #2+2);
\draw[thick] (#1, #2+3) .. controls (#1-1, #2+4) .. (#1, #2+5);
}

\def\SgR (#1,#2){
\fill (#1, #2) circle (3pt);
\fill (#1, #2+2) circle (3pt);
\fill (#1, #2+3) circle (3pt);
\fill (#1, #2+5) circle (3pt);
\draw[thick] (#1, #2+2)--(#1, #2+3);
\draw[thick] (#1, #2) .. controls (#1+1, #2+1) .. (#1, #2+2);
\draw[thick] (#1, #2+3) .. controls (#1+1, #2+4) .. (#1, #2+5);
}

\title{Classifying and Extending $Q_0$-Local $\CA(1)$-Modules}

\author{Katharine Adamyk}
\thanks{Portions of this work were supported by the National Science Foundation under grant No. DMS--1906227.}
\address{Western University, Department of Mathematics, Middlesex College,
London, Ontario, Canada  N6A 5B7}
\email{kadamyk@uwo.ca}

\begin{document}
\maketitle

\vskip.5in

\section*{Abstract}
 In the stable category of bounded below $\CA(1)$--modules, every module is determined by an extension between a module with trivial $Q_0$-Margolis homology and a module with trivial $Q_1$-Margolis homology \cite{Bruner}.  We show that all bounded below $\CA(1)$-modules of finite type whose $Q_1$-Margolis homology is trivial are stably equivalent to direct sums of suspensions of a distinguished family of $\CA(1)$-modules.  Each module in this family is comprised of copies of $\CA(1)\mm\CA(0)$ linked by the action of $Sq^1 \in \CA(1)$.  

The classification theorem is then used to simplify computations of $h_0^{-1}\Ext_{\CA(1)}^{\bullet, \bullet}\big(\emd, \FF_2\big)$ and to provide necessary conditions for lifting $\CA(1)$-modules to $\CA$-modules.  We discuss a Davis--Mahowald spectral sequence converging to $h_0^{-1}\Ext_{\CA(1)}^{\bullet, \bullet}(M, \FF_2)$ where $M$ is any bounded below $\CA(1)$-module.  The differentials in this spectral sequence detect obstructions to lifting the $\CA(1)$-module, $M$, to an $\CA$-module.  We give a formula for the second differential.

\section{Introduction}
\label{intro}

Margolis homology is an invariant of modules over a subalgebra of the Steenrod algebra, the collection of which includes, for example, rings that arise as the total cohomology of a space.  
A classical example of the utility of this invariant is that a module over the Steenrod algebra is free if and only if it has trivial Margolis homology with respect to a particular family of elements \cite[Theorem 19.6]{Margolis}.  An overview of more recent uses is given in \cite{BBt}, including computing the algebra of operations for truncations of the Brown-Peterson spectrum in \cite{AdamsSH} and \cite{Culver}.  In this paper, we focus on Margolis homology for modules over the subalgebra $\CA(1)$, specifically on its uses for classifying $\CA(1)$-modules and determining whether they extend to modules over the entire Steenrod algebra. 

Let $\CA$ be the mod-2 Steenrod algebra, and let $x \in \CA$ satisfy $x^2=0$.   If $M$ is a module over a subalgebra of $\CA$ that contains $x$, the Margolis homology of $M$ with respect to $x$ is the quotient of the kernel of the action of $x$ on $M$ by the image of the same action. 
Over the subalgebra $\CA(1)$, which is generated by $Sq^1$ and $Sq^2$, the relevant types of homology are $Q_0$- and $Q_1$-Margolis homology where $Q_0:=Sq^1$ and $Q_1:=Sq^1 Sq^2 + Sq^2 Sq^1$.  For the most part, we will restrict our attention to $\CA(1)$-modules with trivial $Q_1$-homology. 

In general, restricting to modules with a particular type of Margolis homology can give a subcollection of modules that are more manageable to classify.
For example, Adams and Priddy classified all finitely generated invertible $\CA(1)$-modules by showing they are precisely the finitely generated $\CA(1)$-modules whose $Q_0$- and $Q_1$-homology are both 1-dimensional \cite{AP}. In ongoing work, Fabian Hebestreit and Stephan Stolz have classified bounded $\CA(1)$-modules whose $Q_0$-homology is trivial  and whose $Q_1$-homology is two dimensional. 
This last example concerns modules that are $Q_1$-local---that is, modules with trivial $Q_0$-homology.  The modules we focus on in this paper are, in a sense, of an opposite type; they are $Q_0$-local, meaning they have trivial $Q_1$-homology.
Each $\CA(1)$-module satisfying appropriate finiteness conditions is determined by an extension between a $Q_0$-local module and a $Q_1$-local module in the stable category of $\CA(1)$-modules \cite{Bruner}.  Describing $Q_i$-local modules thus provides some insight into $\CA(1)$-modules in general. 

Here, we provide a classification of $Q_0$-local modules that are bounded below and of finite type.  This does not require any restrictions on the $Q_1$-homology.  To this end, we define a particular family of $Q_0$-local $\CA(1)$-modules.
The first module in this family is the \emph{1-seagull},
\[ \CA(1)\mm\CA(0):= \CA(1) \otimes_{\CA(0)} \FF_2. \] An $n$-seagull, $\Upsilon_{n}$, is a chain of $n$ linked copies of suspensions of $\Upsilon_1$, with $0 < n \leq \infty$.  Section~\ref{secseagulls} contains more detail in Definition~\ref{defnseagulls} and Figure~\ref{figsea}.
We refer to a direct sum of suspensions (i.e., shifts) of seagulls as a flock of seagulls.  Any flock of seagulls is an example of a $Q_0$-local module and we show in the classification theorem they are the only examples, up to free summands.

One use for the classification theorem is to simplify calculations in the Adams spectral sequence.  (The classification theorems of \cite{AP} and Hebestreit--Stolz are both motivated by such calculations.)  Consider the spectrum $ko$, the connective cover of real topological $K$-theory.  This spectrum has cohomology 
\[H^*(ko)\cong \CA \mm \CA(1):= \CA \otimes_{\CA(1)} \FF_2,\] \cite[Part III, Theorem 16.6]{AdamsSH}.   The Adams spectral sequence converging to the 2-completed homotopy groups of $ko$ therefore has the $E_2$-page,
\[ \Ext_{\CA}^{\bullet, \bullet}(\CA \mm \CA(1), \FF_2) \cong \Ext_{\CA(1)}^{\bullet, \bullet}(\FF_2, \FF_2). \]
For any spectrum of the form $ko \wedge X$, we can likewise compute the $E_2$-page of the associated Adams spectral sequence via an equivalent computation over $\CA(1)$.  When $M$ is a module over the entire Steenrod algebra, $\CA$, the term $h_0^{-1}\Ext_{\CA(1)}^{\bullet, \bullet}\big( M, \FF_2 \big)$ can be computed from the Margolis homology of $M$, as a consequence of a theorem of Davis \cite{Davis}.  However, it is not always possible to define an $\CA$-module structure on an arbitrary $\CA(1)$-module in a way that is compatible with the existing action of $\CA(1)$. The first application we give of the classification theorem is a generalization of Davis' theorem for arbitrary $\CA(1)$-modules, utilizing a Davis--Mahowald spectral sequence. 
  
The 1-seagull is  an example of an $\CA(1)$-module that has no compatible $\CA$-module structure as there is no action of $Sq^4$ that is compatible with the relation $Sq^1 Sq^4 + Sq^4 Sq^1=Sq^2 Sq^1 Sq^2$.    One can show similarly that none of the finite seagulls lift to $\CA$-modules.  We give a different proof of this in Section~\ref{seclifting} along with some more general results for lifting $\CA(1)$-modules that are not necessarily $Q_0$-local.

\subsection{Summary of Results}

{The main result of this paper is the classification theorem:}
\begin{classthm}
If $M$ is a bounded below, $Q_{0}$-local $\CA(1)$-module of finite type, then $M$ is stably equivalent to a flock of seagulls.
\end{classthm}
By stably equivalent, we mean isomorphic in the stable category of $\CA(1)$-modules (see Section~\ref{secstablecat}). Equivalently, every $\CA(1)$-module meeting the conditions of Theorem 3.7 is isomorphic to the direct sum of a flock of seagulls and a free module. 

The first application we give of this theorem is to computing localized $\Ext$ terms.  
Let $M$ be a bounded below $\CA(1)$-module of finite type and let $M_*$ denote the dual of $M$ as an $\FF_2$-vector space, with the right $\CA(1)$-module structure given by precomposition by the (left) action of $\CA(1)$ on $M$.  
There exists a spectral sequence, due to Davis and Mahowald \cite{DM},
\[ E_1^{\sigma, s,t} = \Ext_{\CA(0)}^{s,t}(N_\sigma \otimes M, \FF_2) \Rightarrow \Ext_{\CA(1)}^{s,t}(M, \FF_2)\]
where $N_\bullet = \FF_2[x_2, x_3]$.
We consider this spectral sequence when  $h_0 \in \Ext_{\CA(1)}^{1,1}(\FF_2, \FF_2)$ is inverted. 
\begin{e1}
The $E_1$-page of the $h_0$-localized Davis--Mahowald spectral sequence for an $\CA(1)$-module, $M$, is isomorphic, as a trigraded $\FF_2$-vector space, to 
\[H_\bullet(M_*;Q_0)\otimes \FF_2[h_0^{\pm 1}, x_3^2].\]
\end{e1}
We then compute all differentials in the $h_0$-localized Davis--Mahowald spectral sequence associated to a seagull module, $\Upsilon_n$ (Proposition~\ref{propSpSqSeagull}). For an arbitrary bounded below $\CA(1)$-module, $M$, $L_0 M:= \Upsilon_\infty \otimes M$ is $Q_0$-local \cite{Bruner} and thus stably equivalent to a flock of seagulls.  So, in theory, all differentials can be computed in a spectral sequence converging to \[ h_0^{-1}\Ext_{\CA(1)}^{\bullet,\bullet}(L_0 M, \FF_2) \cong h_0^{-1}\Ext_{\CA(1)}^{\bullet, \bullet}(M, \FF_2).\]  However, this relies on computing the decomposition of $\Upsilon_\infty \otimes M$ into seagull modules, which can be labor intensive.  For the first nonzero differential, $d_2$, we are able to give a formula that does not depend on computing this decomposition. 
\begin{d1prop}
The differential 
\[d_2: \FF_2[h_0^{\pm 1}, x_3^2] \otimes H_{\bullet}(M_* ; Q_0)  \to \FF_2[h_0^{\pm 1}, x_3^2] \otimes H_{\bullet}(M_*;Q_0)    \]
is given by
\[ d_2\left(h_0^s x_3^\sigma [b_*]  \right) = h_0^{s-1}x_3^{\sigma+2} \left[ b_* Sq^2 Sq^1 Sq^2  \right]  \]
for all even $\sigma$.
\end{d1prop}  

If $M$ is an $\CA$-module, then the spectral sequence collapses at the $E_1$-page.  So, we are able to use differentials in the spectral sequence to detect obstructions to lifting $\CA(1)$-modules to $\CA$-modules.  The differentials in the spectral sequences associated to the seagull modules lead to the following corollary. 
\begin{liftcor}
Let $M$ be a bounded below $\CA(1)$-module of finite type.
\begin{enumerate}[label=(\roman*)]
    \item If $L_0 M$ is stably equivalent to a flock of seagulls that includes a finite seagull, then $M$ does not lift to an $\CA$-module.
    \item If $M$ is $Q_0$-local, then $M$ lifts to an $\CA$-module if and only if $M$ is stably equivalent to a flock of infinite seagulls (or zero). 
\end{enumerate}\end{liftcor}
Use of this corollary requires determining a decomposition of $L_0 M = \Upsilon_\infty \otimes M$ into seagull modules.   However, if the differentials in the spectral sequence associated to $M$ can be determined directly, they can detect obstructions to lifting without computing this decomposition.  (See, for example Corollary~\ref{cord2lift}.)

\subsection{Organization} Section~\ref{secstablecat} gives a more thorough exposition of relevant background material, including the category of stable $\CA(1)$-modules.  Section~\ref{secseagulls} defines the seagull modules and Section~\ref{secids} proves some helpful identities for working with $Q_0$-local modules.   The proof of the classification theorem begins in Section~\ref{secfinitesg} with the proof that all finite, $Q_0$-local $\CA(1)$-modules are stably equivalent to sums of suspensions of seagulls.  The techniques in this part of the proof do not generalize well to the case of bounded below modules of finite type that are not bounded above. Instead, we show in Section~\ref{secinfinitesg} that any bounded below, $Q_0$-local $\CA(1)$-module of finite type splits as a sum of finite summands, which have been classified, and a summand with no finite summands, which we then show is isomorphic to a sum of suspensions of $\Upsilon_\infty$.

Section~\ref{secapp}, covers the computational applications of the classification theorem.  The main application, discussed in Section~\ref{sech0}, is to computations in a localized Davis--Mahowald spectral sequence.  In Section~\ref{seclifting}, we describe the consequences of this spectral sequence for lifting $\CA(1)$-modules to $\CA$-modules. 

\subsection*{Acknowledgements}
I would like to thank Prasit Bhattacharya, Fabian Hebestreit, Peter May, Haynes Miller, John Palmieri, and Nicolas Ricka for helpful conversations and suggestions. Special thanks to Michael Hopkins, who suggested generalizing Davis' theorem to $\CA(1)$-modules with no compatible $\CA$-module structure and to Bob Bruner and John Rognes who shared their forthcoming work on the Davis--Mahowald spectral sequence and gave helpful suggestions.  Finally, I am especially grateful to Agn\`es Beaudry, who supervised this work as my Ph.D. thesis.

\section{The Classification Theorem}
\label{chapclassification}

In this section, we prove a theorem classifying all bounded below, $Q_0$-local modules of finite type in the stable category of $\CA(1)$-modules.  We begin with some background on the stable category in Section~\ref{secstablecat}.  In Section~\ref{secseagulls}, we define a family of $Q_0$-local modules, named the seagull modules.  After proving some identities in Section~\ref{secids}, we show in Section~\ref{secfinitesg} that every finite $Q_0$-local module is a direct sum of suspensions of seagull modules.  For $Q_0$-local modules that are merely bounded below and finite type, we resolve a few remaining technicalities in Section~\ref{secinfinitesg}. 

\subsection{\texorpdfstring{The Stable Category of 
$\CA(1)$-Modules}{The Stable Category of 
A(1) Modules}}
\label{secstablecat} Unless otherwise specified, all modules are left modules.
In general, we will consider only $\CA(1)$-modules that satisfy certain finiteness conditions, up to a free summand. 
\begin{notation}
For any graded module, $M$, we write $M_k$ for the subgroup of $M$ made up of the homogeneous elements of $M$ of degree $k$.
\end{notation}
We will often restrict our attention to $\CA(1)$-modules satisfying one or both of the following finiteness conditions.  \begin{defn}
An $\CA(1)$-module is bounded below if there exists some $n$ such that $M_k=0$ for all $k \leq n$.  An $\CA(1)$-module is of finite type if $M_k$ is a finite $\FF_2$-vector space for all $k$. 
\end{defn}

The benefit of considering appropriately finite $\CA(1)$-modules, up to a free summand, is that it allows us to work over the stable category of (bounded below) $\CA(1)$-modules, which has nice properties with respect to Margolis homology. 
\begin{defn}
The stable category of $\CA(1)$-modules, $\Stab\big(\CA(1)\big)$, be the category with all $\CA(1)$-modules as objects and morphisms $[M, N] = \Hom_{\CA(1)}(M, N)/\sim$ where $f\sim g$ if $f-g$ factors through a free $\CA(1)$-module.
\end{defn}
We say two $\CA(1)$-modules are stably equivalent if they are isomorphic in $\Stab\big(\CA(1)\big)$.
Over $\CA(1)$ (or, in fact, any $\CA(n)$), a module is free if and only if it is projective \cite[Proposition 12.2.8]{Margolis}.  So, this is equivalent to the usual definition of a stable module category, where maps are identified if their difference factors through a projective module (see for example, Chapter 2.2 of \cite{Hovey}).

Let $\Stab\big(\CA(1)\big)^b$ be the full subcategory of $\Stab\big( \CA(1)\big)$ whose objects are bounded below $\CA(1)$-modules. 
Every module in $\Stab\big(\CA(1)\big)^b$ splits as a direct sum of a free module and a module with no free summands \cite{Margolis}. 
\begin{defn}
Let $M = M_{red} \oplus M_{free}$ where $M_{free}$ is free and $M_{red}$ has no free summands.  The module $M_{red}$ is called the reduced part of $M$, and $M$ is reduced if $M_{red} \cong M$.
\end{defn}
Two bounded below $\CA(1)$-modules are stably equivalent if and only if their reduced parts are isomorphic in $\CA(1)\Mod$ \cite{Margolis}.

Let $Q_0=Sq^1$ and $Q_1=Sq^2 Sq^1 + Sq^1 Sq^2$.  Since $Q_i^2=0$,  $Q_i$-Margolis homology can be defined as \[H_k(M;Q_i) = \dfrac{\ker(Q_i)_k}{\im(Q_i)_k}\] for any $\CA(1)$-module, $M$.  Any module that is free over $\CA(1)$ has trivial $Q_0$- and $Q_1$-homology, so if $M$ and $N$ are stably equivalent $\CA(1)$-modules, they have the same $Q_i$-homology.  
The stable category of $\CA(1)$-modules is triangulated, and
every $M \in \Stab\big(\CA(1)\big)^b$ sits in a unique triangle,
\[ L_0 M \to M \to L_1 M \to L_0 M [1], \]
where $L_i M$ is $Q_i$-local \cite{Bruner}.
The module $L_i M$ is called the $Q_i$-localization of $M$.  
To understand $\CA(1)$-modules more broadly, it would therefore be helpful to know more about $Q_i$-local modules.  

\subsection{The Seagull Modules}\label{secseagulls}
In this section, we define the seagull modules, which are depicted in Figure~\ref{figsea}.
The first seagull module is
\[ \Upsilon_1 := \CA(1)\mm\CA(0) = \CA(1) \otimes_{\CA(0)}\FF_2. \]
Note that $\CA(1) \otimes_{\CA(0)} \FF_2$ has the $\CA(1)$-module structure given by acting on the left factor. (This is in contrast to the $\CA(1)$-module structure on a tensor product over $\FF_2$ of $\CA(1)$-modules which is given by the diagonal action.)  An $n$-seagull, $\Upsilon_{n}$, is a chain of $n$ linked copies of suspensions of $\Upsilon_1$.

\begin{defn}\label{defnseagulls}  The $n$-seagull, $\Upsilon_{n}$, is generated as an $\CA(1)$-module by $\{y_{4j}\}_{0 \leq j \leq n-1}$ with $|y_{4j}|=4j$ where $Sq^2 Sq^1 Sq^2 y_{4j} \neq 0$ for all $j$ and 
\[ Sq^{1}y_{4j}= \begin{cases} 0 & j=0 \\ Sq^2 Sq^1 Sq^2 y_{4(j-1)} & j>0. \end{cases} \]
\end{defn}

\begin{figure}[ht!]
    \caption{The modules $\Upsilon_n$. The bottom class of each module is in degree zero.}
    \label{figsea}
    \phantom{.} \hfill
  \begin{tikzpicture}[scale=0.3] 
    \SgL(0,0);
    \end{tikzpicture} 
    \hfill 
        \begin{tikzpicture}[scale=0.3] 
    \SgL(0,0);
    \SgR(0,4);
    \sqone(0,4,);
    \end{tikzpicture} 
    \hfill 
        \begin{tikzpicture}[scale=0.3] 
    \SgL(0,0);
     \SgR(0,4);
    \sqone(0,4,);
    \SgL(0,8);
    \sqone(0,8,);
    \end{tikzpicture} 
    \hfill $\cdots$  \hfill 
    \begin{tikzpicture}[scale=0.3] 
    \SgL(0,0);
     \SgR(0,4);
    \sqone(0,4,);
    \SgL(0,8);
    \sqone(0,8,);
    \sqone(0,12,);
    \sqtwoR(0,12,);
    \fill(0,12) circle(3pt);
    \fill(0,14) circle (3pt);
    \draw[thick] (0,14) -- (0,14.5);
    \fill(0,15) circle (2pt);
    \fill(0,15.5) circle (2pt);
    \fill(0,16) circle (2pt);
    \end{tikzpicture} \hfill  \phantom{.} \\
    \phantom{.} \hfill $\Upsilon_1$ \hfill $\Upsilon_2$ \hfill $\Upsilon_3$ \hfill \hfill   $\Upsilon_\infty$ \hfill \phantom{.}
\end{figure}
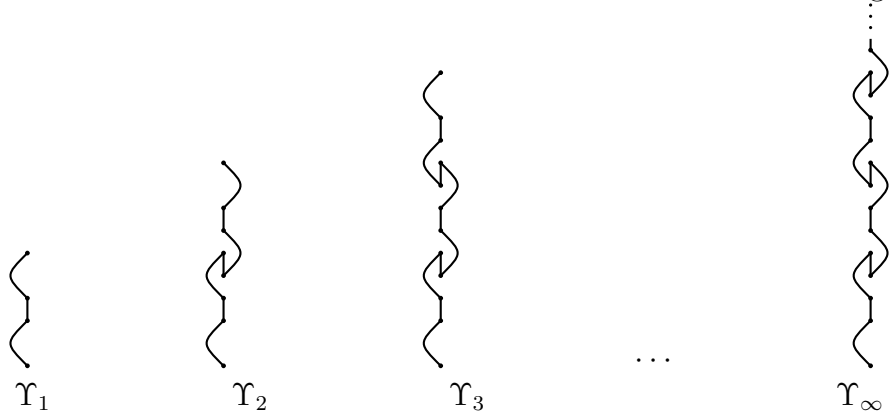

We call a direct sum of suspensions of seagulls of various lengths, 
\[ \bigoplus_{i \in I} \Sigma^{\alpha_i} \Upsilon_{n_i}, \]
a flock of seagulls. 
We sometimes refer to any suspension of an $n$-seagull as an $n$-seagull, but $\Sigma^k \Upsilon_n$ will always mean the $k^{\text{th}}$ suspension of the $n$-seagull whose first generator is in degree zero. 
\begin{notation}
When working in categories of modules over a subalgebra of $\CA$ we always use ``suspension'' and the character $\Sigma$ to mean a shift in the grading of the module.  In particular, we do not use $\Sigma$ for  the translation functor $M \mapsto M[1]$ in the triangulated category $\Stab^b\big(\CA(1)\big)$.
\end{notation}

For all $n$, the Margolis homology of $\Upsilon_{n}$ is given by
\begin{align*}
    H_{\bullet}(\Upsilon_n;Q_0) &= \begin{cases} \FF_2 \oplus \Sigma^{4n+1}\FF_2 & n<\infty \\ \FF_2 & n=\infty \end{cases}\\
    H_{\bullet}(\Upsilon_n;Q_1) &= 0.
\end{align*}
So the $n$-seagulls are $Q_0$-local. Furthermore, taking the tensor product of a bounded below $\CA(1)$-module, $M$, and the infinite seagull, $\Upsilon_\infty$, results in the $Q_0$-localization of $M$ \cite{Bruner}.  Note that $\Upsilon_\infty \otimes M$ is an $\CA(1)$-module via the diagonal action. 

The seagull modules are not only examples of $Q_0$-local modules, they are essentially the only examples. 
\begin{thm}\label{thmclassification} If $M$ is a bounded below,  $Q_{0}$-local $\CA(1)$-module of finite type, then $M$ is stably equivalent to a flock of seagulls. 
\end{thm}
The remainder of Section~\ref{chapclassification} concerns the proof of this theorem.

\subsection{\texorpdfstring{Properties of $Q_0$-local $\CA(1)$-modules}{Properties of Q0-local A(1)-modules}}
\label{secids}

This section provides several properties of reduced, bounded below, $Q_0$-local $\CA(1)$-modules that will be needed in proving the classification theorem. In the rest of this section, $M$ will always denote a reduced, bounded below, $Q_0$-local $\CA(1)$-module.   After possibly shifting $M$, we assume $M_k$ is nonzero for $k=0$ and zero for all $k<0$ (such a module will be referred to as connective).

In this section, we will frequently need the following relations that hold in $\CA(1)$:
\begin{align}
    Sq^1 Sq^1 &= 0 \label{eq1} \\
    Sq^2 Sq^2 &= Sq^1 Sq^2 Sq^1 \label{eq2} \\
    Sq^1 Sq^2 Sq^2 &= Sq^2 Sq^2 Sq^1 =0 \label{eq3}\\
    Sq^1 Sq^2 Sq^1 Sq^2 &= Sq^2 Sq^1 Sq^2 Sq^1. \label{eq4}
\end{align}
The relations (\ref*{eq3}) and (\ref*{eq4}) are generated by (\ref*{eq1}) and (\ref*{eq2}) (as are all relations in $\CA(1)$).
We will also repeatedly use the fact that $\ker(Q_1)\cap M_k$ is zero for all $k \leq 2$.  To see this, note that $H_\bullet (M;Q_1)=0$ and that $\im(Q_1)\cap M_k=0$ in degrees $k \leq 2$. 

\begin{lemma}\label{lemidentities}
Let $M$ be a reduced, connective, $Q_0$-local $\CA(1)$-module and let $x$ be any nonzero element of $M_0$. Then, 
\begin{enumerate}[label=\roman*.)]
\item $Sq^{2}Sq^{1}Sq^{2}x \neq 0 $, and
\item $Sq^{1}Sq^{2}Sq^{1}Sq^{2}x = Sq^{2}Sq^{1}Sq^{2}Sq^{1}x = 0$.
\end{enumerate}
\end{lemma}
\begin{proof}
We assume $M_k=0$ for all $k<0$ and $M_k \neq 0$. Let $x \in M_0$ be nonzero.  Then $Q_1 x \neq 0$, since $\ker Q_1=0$ in degree zero.  If $Sq^{1}x \neq 0$, then \[Q_1(Sq^{1}x)=Sq^{2}Sq^{1}Sq^{1}x + Sq^{1}Sq^{2}Sq^{1}x = Sq^{1}Sq^{2}Sq^{1}x=Sq^{2}Sq^{2}x\] must also be nonzero. So, $Sq^{2}x \neq 0$ when $Sq^1 x \neq 0$.  If $Sq^{1}x=0$, then $Q_1 x = Sq^{1}Sq^{2}x$ is nonzero, so we can conclude $Sq^{2}x\neq 0$ in this case as well. 

By virtue of degree, $Sq^{2}x \notin \text{ker }Q_{1}$.  So, 
\begin{align*} Q_{1}(Sq^{2}x) &= Sq^{2}Sq^{1}Sq^{2}x + Sq^{1}Sq^{2}Sq^{2}x  = Sq^{2}Sq^{1}Sq^{2}x  \neq 0 \end{align*}
where the second equality follows from Equation~\ref{eq3}. This proves (i).   
If $Sq^{1}Sq^{2}Sq^{1}Sq^{2}x \neq 0$, then $x$ supports a free submodule of $M$. {Since free $\CA(1)$-modules are injective over $\CA(1)$ \cite[Proposition 12.2.8]{Margolis}, this implies $x$ supports a free direct summand of $M$.} So, (ii) follows from the fact that $M$ is reduced. 
\end{proof}
\begin{lemma}\label{lemsq1=0}
Let $M$ be a reduced, connective, $Q_{0}$-local $\CA(1)$-module and let $x$ be any element of $M_{0}$.  Then $Sq^{1}x=0$.  
\end{lemma}
\begin{proof}
Suppose $Sq^{1}x\neq 0$. Then $Sq^{1}x \notin \ker(Q_{1})$, since $|Sq^{1}x|=1$. 
So,
\begin{align}\label{eqsq1sq2sq1} 0 \neq Q_{1}(Sq^{1}x)= Sq^{2}Sq^{1}(Sq^{1}x)+Sq^{1}Sq^{2}(Sq^{1}x) = Sq^{1}Sq^{2}Sq^{1} x  \end{align}
and thus $Sq^{2}Sq^{1}x$ is nonzero. 
Using Lemma~\ref{lemidentities} and Equation~\ref{eq3}, we see
\[ Q_1(Sq^2 Sq^1 x) = Sq^{1}Sq^{2}Sq^{2}Sq^1 x + Sq^{2}Sq^{1}Sq^{2}Sq^{1}x = Sq^{2}Sq^{1}Sq^{2}Sq^{1}x = 0. \]
Since $H_\bullet (M;Q_1)=0$, $Sq^{2}Sq^{1}x$ must be in $\im(Q_{1})$. Let $z\in M_{0}$ be such that $Q_{1}z=Sq^{2}Sq^{1}x$.  Since $|z|=0$, Lemma~\ref{lemidentities} implies that $Sq^2 z \neq 0$. In the remainder of this proof, we will show $Q_1(Sq^2 z)=0$. This is a contradiction, as $Sq^2 z$ is in degree two and therefore not in the image of $Q_1$.  

Assuming $Q_1 z=Sq^2 Sq^1 x$,
\[Sq^{2}Sq^{1}z + Sq^{1}Sq^{2}z = Sq^{2} Sq^{1} x.\]
Applying $Sq^1$ to this equation, it follows that $Sq^1 Sq^2 Sq^1 z = Sq^1 Sq^2 Sq^1 x.$
So, using Equation~\eqref{eqsq1sq2sq1},
    \begin{align*}
    Sq^{1}Sq^{2}Sq^{1}z  \neq 0.
\end{align*}
Then $Sq^{1}z \neq 0$ and $Sq^{2}z \neq 0$, since $Sq^{2}Sq^{2}z=Sq^{1}Sq^{2}Sq^{1}z$. Next, note that 
\begin{align*}
    Q_{1}(Sq^{1}x+Sq^{1}z)&= Sq^{1}Sq^{2}Sq^{1}x+Sq^{1}Sq^{2}Sq^{1}z 
= 0.
\end{align*}
Since $\ker(Q_1) \cap M_1$ is zero, $Sq^{1}x=Sq^{1}z$.
Then, 
\begin{align*}
    Sq^{2}Sq^{1}x &= Q_1 z = Sq^{2}Sq^{1}z + Sq^{1}Sq^{2}z =Sq^{2}Sq^{1}x + Sq^{1}Sq^{2}z.
\end{align*}
Thus, $Sq^{1}Sq^{2}z =0$.
But then,
\[ Q_{1}(Sq^{2}z) = Sq^{2}Sq^{1}Sq^{2}z + Sq^{1}Sq^{2}Sq^{2}z = 0, \]
where we again use Equation~\ref{eq3}. So, we have reached the desired contradiction. 
\end{proof}
The final lemma we will need applies specifically to the seagull modules. 
\begin{lemma}\label{lemker1_im2}
If $N$ is a flock of seagulls, 
\[ \im\left(Sq^2\right) \cap \ker\left(Sq^1\right) = \im\left(Sq^2 Sq^1 Sq^2\right)\]
in $N$.
\end{lemma}
\begin{proof}
Let 
\[ N = \bigoplus_{i \in I} \Sigma^{\alpha_i} \Upsilon_{n_i} \]
where $\Sigma^{\alpha_i}\Upsilon_{n_i}$ is generated by $ y_{\alpha_i}, y_{\alpha_i+4}, \ldots, y_{\alpha_i+4(n_i-1)}$.  For any $x \in \CA(1)$, let
\[ N(x):= \bigoplus_{\substack{i \in I \\ j \in J_i}} \FF_2\{ x (y_{i,j}) \} \]
where $J_i =\{\alpha_i+4s | 0 \leq s \leq n_i-1\}$.

As an $\FF_2$-vector space,
\[ N \cong N(1) \oplus N(Sq^2) \oplus N(Sq^1 Sq^2) \oplus N(Sq^2 Sq^1 Sq^2).  \]
So,
\[ \ker(Sq^1) \cong N(1) \oplus N(Sq^1 Sq^2) \oplus N(Sq^2 Sq^1 Sq^2), \]
and,
\[ \im(Sq^2) \cong N(Sq^2) \oplus N(Sq^2 Sq^1 Sq^2).  \]
Examining the intersection, the result follows. 
\end{proof}

\subsection{Classification of Finite Modules}\label{secfinitesg}
Having proven some basic properties of reduced, $Q_0$-local $\CA(1)$-modules, we proceed to the proof of the main theorem in the case where $M$ is finite.  The goal of this section is to prove the following result.
\begin{prop}\label{propfiniteclass}   If $M$ is a reduced, finite, $Q_{0}$-local $\CA(1)$-module, then $M$ is isomorphic to a flock of seagulls.  \end{prop}
Note that finite is equivalent to bounded and finite type, so this differs from Theorem~\ref{thmclassification} in that we also require $M$ to be bounded above. 
\begin{notation}
For any graded module, $M$, we use $M^k$ to denote the smallest submodule of $M$ containing all homogeneous elements with degree no more than $k$. 
\end{notation}
This is in contrast to $M_k$, which continues to represent the elements of $M$ in degree exactly $k$. 
We will prove Proposition~\ref{propfiniteclass} via induction on $M^k$.
\begin{lemma}\label{lembasecase}
If $M$ is a reduced, finite, $Q_0$-local $\CA(1)$-module, then $M^0$ is isomorphic to a flock of 1-seagulls.
\end{lemma}
\begin{proof}
Let $\CB$ be an $\FF_{2}$ basis for $M_{0}$.  Then $\CB$ is a minimal $\CA(1)$ generating set for $M^{0}$.  We use the notation  $Sq^R = Sq^{r_1}Sq^{r_2}\cdots$ for any sequence $R=(r_1, r_2, \ldots)$.  For each $b \in \CB$, the map $f_{b}: \Upsilon_{1}\{y_{b}\} \to M$ given by $Sq^R y_b \mapsto Sq^R b$ is an $\CA(1)$-map because $Sq^1 b =0$ by Lemma~\ref{lemsq1=0}.  Appealing to Lemma~\ref{lemidentities}, we see this map is also an injection.  
We claim the direct sum of the $f_b$'s, 
\begin{align*}
    f :\bigoplus_{b \in \CB} \Upsilon_{1}\{y_{b}\} &\to  M 
\end{align*}
is an isomorphism. 

Since $f$ surjects onto an $\CA(1)$ generating set of $M^{0}$ and is an $\CA(1)$-module map, $f$ surjects onto $M^{0}$.  
Restricting to degree zero, $f$ is an isomorphism between the $\FF_2$ vector spaces $\FF_2\{y_b | b \in \CB\}$ and $\FF_2\{b | b \in \CB\}$. Let $y$ be a homogenous element of $\bigoplus_{b \in \CB} \Upsilon_1 \{y_b\}$ in any degree. Then $y=Sq^{R}z$ where $z$ is a linear combination of the $y_b$'s and $R \in \{(0), (2), (1,2), (2,1,2)\}$.  If $y \in \ker(f)$, then $0=f(y)=Sq^R f(z)$.  No matter what $R$ is, this implies $Sq^{2}Sq^{1}Sq^{2}f(z)=0$.  By Lemma~\ref{lemidentities}, $f(z)=0$.  Since $f$ is an isomorphism of $\FF_2$ vector spaces in degree zero, $z=0$.  Thus, $y=Sq^R z =0$ and $f$ is injective. 
\end{proof}

The rest of this section will prove that if $M^{k-1}$ is isomorphic to a flock of seagulls, then $M^k$ is also isomorphic to a flock of seagulls. We will show in Lemma~\ref{lemquotient} that $M^k/M^{k-1}$ is a flock of 1-seagulls.  So, identifying $M^k$ amounts to solving an extension problem between flocks of seagulls. 

Letting 
 \[M^{k-1} \cong \bigoplus_{i \in I} \Sigma^{\alpha_{i}} \Upsilon_{n_{i}}, \]
we will use the notation that $\Sigma^{\alpha_i}\Upsilon_{n_i}$ is generated by
$y_{i, \alpha_{i}}, y_{i, \alpha_{i}+4}, \ldots, y_{i, \alpha_{i}+4(n_{i}-1)}$,
so $y_{i,j}$ is in the $i^{\text{th}}$ seagull and has degree $j$.  Note that for all $i$, the degree of the top generator in the $i^\text{th}$-seagull, $\alpha_{i}+4(n_{i}-1)$, is no more than $k-1$.

\begin{lemma}\label{lemquotient}
If $M$ is a reduced, bounded below, $Q_0$-local $\CA(1)$-module and $M^{k-1}$ is isomorphic to a flock of seagulls, then $M^k/M^{k-1}$ is isomorphic to a sum of 1-seagulls. 
\end{lemma}
\begin{proof}
By assumption, $M^{k-1}$ is reduced and $Q_{0}$-local. Since $M$ is also reduced, the quotient $M/M^{k-1}$ is reduced.  The short exact sequence
\[ 0 \to M^{k-1} \to M \to M/M^{k-1} \to 0 \]
induces a long exact sequence in $Q_1$-homology, so $M/M^{k-1}$ is $Q_0$-local.
Then, we can apply Lemma~\ref{lembasecase} to $\Sigma^{-k}(M/M^{k-1})$. Thus $(M/M^{k-1})^{k}=M^{k}/M^{k-1}$ is isomorphic to a sum of 1-seagulls, each generated by a class in degree $k$. 
\end{proof}
In order to describe the extension,
\[ 0 \to M^{k-1} \to M^k \to M^k/M^{k-1} \to 0 \]
it will be useful to fix an $\FF_2$ basis for $M_k$.
\begin{defn}\label{defnb}
Let $\CB$ be an $\FF_2$-basis for $\big(M^{k-1}\big)_k$ and let $K$ be the kernel of the $Sq^{1}$ action on $M$.  Then $\big(M^{k-1}\big)_k$ is a subspace of $\big(M^{k-1} + K \big)_k$ and $\CB$ can be completed to a basis for $\big(M^{k-1} + K \big)_k$ given by $\CB \cup \{b_1, \ldots, b_\ell\}$ for some $b_i \in M^k$. Complete this basis for $(M^{k-1} + K)_k \subseteq M_k$ to obtain a basis for $M_k=\big(M^{k}\big)_k$ given by $\CB\cup\{b_1, \ldots, b_\ell, b_{\ell+1}, \ldots, b_d\} $.  

Let $[b_i]$ be the class containing $b_i \in M^k$ in the quotient $M^k/M^{k-1}$.
Note that $\{[b_{1}], \ldots, [b_{\ell}]\}$ is a basis for $(K/(K \cap M^{k-1}))_k$ and $\{[b_{1}], \ldots,  [b_{d}]\}$ is a basis for $(M^k/M^{k-1})_k$.  Since $M^k/M^{k-1}$ is a flock of 1-seagulls, each $b_i$ satisfies $Sq^{2}Sq^{1}Sq^{2}b_i \neq 0$ and $Sq^{1}b_i \in M^{k-1}$.
\end{defn}
\begin{rmk}\label{rmkquotientbasis}
It follows that \[ M^k/M^{k-1} \cong \bigoplus_{1 \leq i \leq d} \Sigma^k \Upsilon_1\{[b_i]\}. \]
\end{rmk}

\begin{lemma} \label{lemplus1sgs}
The submodule of $M$ generated by $M^{k-1}$ and $\{b_{1}, \ldots, b_{\ell}\}$ is isomorphic to 
\[M^{k-1} \oplus \left(\bigoplus\limits_{1 \leq i \leq \ell} \Sigma^k \Upsilon_{1}\{w_{i}\} \right).\]\end{lemma}

\begin{proof}  
Let $Z$ be the submodule of $M$ generated by $M^{k-1}$ and $\{b_{1}, \ldots, b_{\ell}\}$, let
\[Y=M^{k-1} \oplus \left(\bigoplus\limits_{1 \leq i \leq \ell} \Sigma^k \Upsilon_{1}\{w_{i}\} \right),\]
and let 
\[\Psi: M^{k-1}\oplus \left(\bigoplus\limits_{1 \leq i \leq \ell} \Sigma^{k}\Upsilon_{1}\{w_{i}\} \right) \to M^{k}\]
be the $\CA(1)$-map determined by inclusion on the left summand and mapping $w_i$ in the right summand to $b_{i}$. 
On the right summand, we use the fact that $Sq^1 b_i =0$ for all $1\leq i \leq \ell$ to conclude that there is such an $\CA(1)$-map.  

So, we have the diagram of $\CA(1)$-modules,
\[ \begin{tikzcd} 
0 \ar[r] & M^{k-1} \ar[r] \ar[d, "\cong" turn] & Y \ar[r] \ar[d, "\Psi"] & \bigoplus\limits_{1 \leq i \leq \ell} \Sigma^{k} \Upsilon_{1}\{w_{i}\} \ar[r] \arrow[d, "\cong" turn] & 0 \\
0\ar[r] & M^{k-1} \ar[r] & Z \ar[r] & \bigoplus\limits_{1 \leq i \leq \ell} \Sigma^{k} \Upsilon_{1}\{b_{i}\} \ar[r] & 0.
\end{tikzcd} \]
The rows are exact, the map from $M^{k-1}$ to itself is the identity, and the righthand map takes $w_{i}$ to $b_{i}$.  The construction of $\Psi$ guarantees this diagram commutes.  By the 5-lemma, $\Psi$ is thus an isomorphism.  
\end{proof}

Now we have to consider the rest of the new generators.  We will start by showing we can pick the new generators in a way that will make solving the extension problem simpler. 

\begin{lemma}\label{lemchoose_b}
The generators $\{b_{\ell+1}, \ldots, b_d\}$ can be chosen so that $Sq^1 b_i \in \im(Sq^2 Sq^1 Sq^2)$ for all $\ell+1 \leq i \leq d$.
\end{lemma}
\begin{proof}
Since $\left[Sq^1 b_i \right]=0$ in $M^k/M^{k-1}$, $Sq^1 b_i \in (M^{k-1})_{k+1}$. All generators of $M^{k-1}$ are in degree no more than $k-1$, so
\[ Sq^1 b_i = Sq^1 a_i + Sq^2 c_i \]
for some $a_i, c_i \in M^{k-1}$.  While $Sq^1 a_i$ may be zero, $Sq^2 c_i$ is nonzero since otherwise $a_i + b_i$ would be in $K_k \subseteq \mathrm{span}\left\{\CB \cup \{b_1, \ldots, b_\ell \}\right\}$.

Note that 
\[ Sq^1 (Sq^2 c_i ) = Sq^1(Sq^1 b_i + Sq^1 a_i) = 0. \]
Since $M^{k-1}$ is a flock of seagulls, Lemma~\ref{lemker1_im2} says 
\[ \ker(Sq^1) \cap \im(Sq^2) = \im(Sq^2 Sq^1 Sq^2) \]
in $M^{k-1}$.

So, $Sq^2 c_i \in \im(Sq^2 Sq^1 Sq^2)$. Let $\widehat{b}_i=b_i+a_i$.  Then $Sq^1 \wh{b}_i = Sq^2 c_i \in \im(Sq^2 Sq^1 Sq^2)$ and $\{b_{\ell+1}, \ldots, b_{i-1}, \wh{b}_i, b_{i+1}, \ldots, b_d\}$ satisfies:
\begin{itemize}
    \item $\CB \cup \{b_{1}, \ldots, b_{i-1}, \wh{b}_i, b_{i+1}, \ldots, b_d\}$ is a basis for $M_k$ since $a_i \in \mathrm{span}\{\CB\}$, 
    \item $\CB \cup \{b_{1}, \ldots, b_\ell\}$ is a basis for $(M^{k-1}+K)_k$ since $b_i$ has not been changed for $i\leq \ell$, and 
    \item $\{[b_{\ell+1}], \ldots, [b_{i-1}], [\wh{b}_i], [b_{i+1}], \ldots, [b_d]\}$ is a basis for $(M^k/M^{k-1})_k$ since $[b_i]=[\wh{b}_i]$. \qedhere
\end{itemize}
\end{proof}

For any $\ell \leq \lambda \leq d$, let $M(\lambda)$ be the submodule of $M$ generated by $M^{k-1}$ and $\{b_{1}, \ldots, b_{\lambda}\}$.  In Lemma~\ref{lemplus1sgs} $M(\ell)$, we showed $M(\ell)$ is isomorphic to a flock of seagulls.  Suppose $M(\lambda-1)$ is isomorphic to a sum of seagulls, $\bigoplus\limits_{i \in I_{\lambda-1}} \Sigma^{\alpha_{i}}\Upsilon_{n_{i}}$.  We'll label the generators of $\Sigma^{\alpha_{i}}\Upsilon_{n_{i}}$ as $y_{i, \alpha_{i}}, y_{i, \alpha_{i}+4}, \ldots, y_{i, \alpha_{i}+4(n_{i}-1)}$.
Note that $M(\lambda-1)\subseteq M^k$, so $\alpha_i+4(n_i-1)$, the degree of the top generator in the $i^{th}$ seagull, is no more than $k$. 

We will be interested not just in the fact that $M(\lambda)$ is a flock of seagulls, but in how the flock of seagulls has changed from $M(\lambda-1)$. At each stage, we identify each seagull in the flock $M(\lambda-1)$ as one of the following types, depicted in Figure~\ref{figLAU}
\begin{itemize}
    \item A \textbf{lengthened seagull} has a generator in degree $k$ and a generator in degree $k-4$.  This type of seagull is the result of a new 1-seagull being attached to a sum of seagulls in $M^{k-1}$ in a previous step.  The lengthened seagulls in $M(\lambda-1)$ will be indexed by $L_{\lambda-1}$.
    \item An \textbf{available seagull} has a generator in degree $k-4$, but no generator in degree $k$. So, an available seagull has a top class in degree $k+1$ that is not in the image of $Sq^1$ acting on $M(\lambda-1)$. The available seagulls in $M(\lambda-1)$ will be indexed by $A_{\lambda-1}$.
    \item An \textbf{unavailable seagull} does not have a generator in degree $k-4$.  This includes all the 1-seagulls with generators in degree $k$ that were included in $M(\ell)$.  The unavailable seagulls in $M(\lambda)$ will be indexed by $U_\lambda$.
\end{itemize}

We make this precise in the following definition.
\begin{defn}\label{defnLAU}
Let 
\begin{align*}
    L_{\lambda-1} :=& \{i \in I_{\lambda-1} | \alpha_i+4(n_i-1) = k, n_i>1\} \\
    A_{\lambda-1} :=& \{i \in I_{\lambda-1} | \alpha_i+4(n_i-1) = k-4  \} \\
    U_{\lambda-1} :=& I_{\lambda-1} \setminus (L_{\lambda-1} \cup A_{\lambda-1}).
\end{align*}
\end{defn} 

\begin{figure}[ht!]
    \centering
    \caption{The decomposition of $M(\lambda-1)$ into available, lengthened, and unavailable seagulls.}
    \label{figLAU}
    \begin{tikzpicture}[scale=0.35]
        \draw[gray, dashed] (-18,10)--(10,10);
        \node[left] at (14,10){\text{deg. }$k$};
        \SgL(8,0);
        \SgR(8,4);
        \sqone(8,4,);
        \SgL(8,8);
        \sqone(8,8,);
        \SgR(2,3);
        \SgR(4,7);
        \SgL(6,10);
        \SgL(0,5);
        \SgR(0,1);
        \sqone(0,5,);
        \SgL(-16,6);
        \SgR(-16,2);
        \sqone(-16,6,);
        \SgL(-14,6);
        \SgR(-12,-2);
        \SgL(-12,2);
        \sqone(-12,2,);
        \SgR(-12,6);
        \sqone(-12,6,);
        \SgL(-7,10);
        \SgR(-7,6);
        \sqone(-7,6,);
        \SgL(-7,2);
        \sqone(-7,10,);
        \SgR(-5,10);
        \SgL(-5,6);
        \sqone(-5,10,);
        \draw[gray](-17,3)--(-17,-3)--(-10,-3)--(-10,-1);
        \draw[gray](-13.5,-3)--(-13.5,-4);
        \node[below] at (-14,-4){Available};
        \draw[gray](-9,3)--(-9,1)--(-4,1)--(-4,7);
        \draw[gray](-6.5,1)--(-6.5,0);
        \node[below] at (-6,0){Lengthened};
        \draw[gray](-1,2)--(-1,-1)--(9,-1)--(9,1);
        \draw[gray](4,-1)--(4,-2);
        \node[below] at (4,-2){Unavailable};
    \end{tikzpicture}
\end{figure}
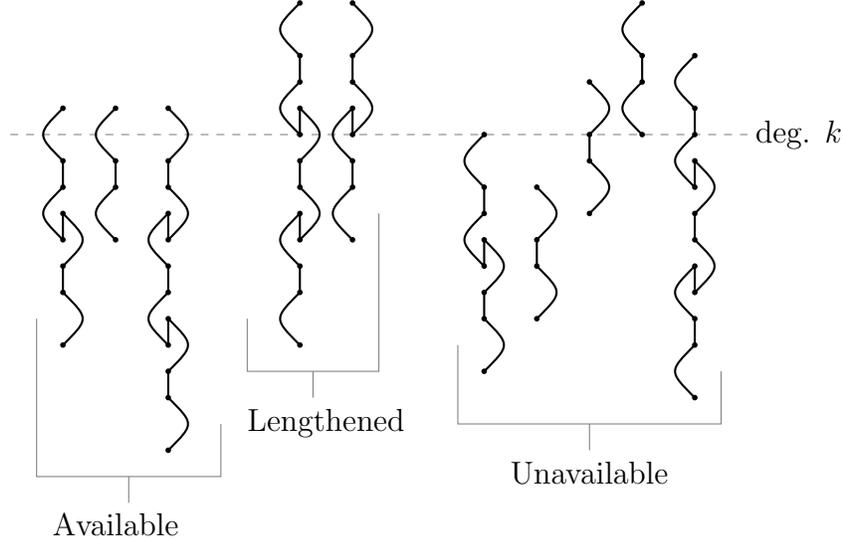

  The general idea of the rest of this section is that each time we move from $M(\lambda-1)$ to $M(\lambda)$, we would like to attach a 1-seagull to the top of an available seagull, creating a new lengthened seagull.
  However, the actual process is a little more complicated than this, as we may have to change the basis for $M(\lambda-1)$ so that $Sq^1$ on the new generator hits the top of a single available seagull.
  We will start by showing that the new generator, $b_\lambda$, can be chosen so that $Sq^1 b_\lambda$ is a sum of classes in available seagulls.  

\begin{lemma}\label{lemb_available}
The generator $b_{\lambda}$ can be chosen so that \[Sq^1 b_\lambda = Sq^2 Sq^1 Sq^2\sum\limits_{p \in P} y_{p,k-4}\] where $P \subseteq A_{\lambda-1}$.
\end{lemma}
\begin{proof}
We chose the $b_i$'s so that $Sq^1 b_i \in \im(Sq^2 Sq^1 Sq^2)_{k+1}$ for all $i \geq \ell$ (Lemma~\ref{lemchoose_b}).  
This means we can express $Sq^1 b_\lambda$ as a sum,
\[ Sq^1 b_\lambda = Sq^2 Sq^1 Sq^2 \sum_{i \in I'} y_{i, k-4}, \]
where $I' \subseteq I_{\lambda-1}$.
Unavailable seagulls do not have a nonzero class in the image of $Sq^2 Sq^1 Sq^2$ in degree $k+1$, so $I' \subseteq A_{\lambda-1}\cup L_{\lambda-1}$. 

Let $I''=I' \cap A_{\lambda-1}$ (that is, $I''$ indexes the available seagulls whose top class is hit by $Sq^1 b_\lambda$).  We claim $I''$ is nonempty.  
If not, then all seagulls indexed by $I'$ are lengthened, so they each have a generator, $y_{i,k-4}$.  This generator is linked to the generator in degree $k$ by the relation, $Sq^2 Sq^1 Sq^2 y_{i, k-4} = Sq^1 y_{i, k}$, and thus
\[ Sq^1 b_\lambda = \sum_{i \in I'} Sq^2 Sq^1 Sq^2 y_{i, k-4} = \sum_{i \in I'} Sq^1 y_{i, k}. \]
In this case, $b_\lambda+\sum\limits_{i \in I'} y_{i,k}$ is in $K$, the kernel of the $Sq^1$ action on $M$.  Then $\sum\limits_{i \in I'} y_{i,k}$ and  $b_\lambda+\sum\limits_{i \in I'} y_{i,k}$ are both in $(M^{k-1}+K)_k$, the $\FF_2$ subspace of $(M^k)_k$ generated by $(M^{k-1})_k$ and $b_1, \ldots, b_{\ell}$.  However, their sum, $b_\lambda$, is not in $(M^{k-1}+K)_k$ by construction. So, $I''$ must be nonempty.  

If we replace $b_\lambda$ in the basis for $(M^k)_k$ with \[\wh{b}_\lambda=b_\lambda + \sum_{i \in I'\setminus I''}y_{i,k},\] then \[Sq^1 \wh{b}_\lambda = Sq^2 Sq^1 Sq^2 \sum_{i \in I''} y_{i, k-4}\] and the following properties are preserved:
\begin{itemize}
    \item $\CB \cup \{b_{1}, \ldots, b_{\lambda-1}, \wh{b}_\lambda, b_{\lambda+1}, \ldots, b_d\}$ is a basis for $M_k$ since $b_\lambda + \wh{b}_\lambda$ is in $\big(M(\lambda-1)\big)_k$ which is spanned by $\CB \cup \{b_1, \ldots, b_{\lambda-1}\}$,
    \item $\CB \cup \{b_{1}, \ldots, b_\ell\}$ is a basis for $(M^{k-1}+K)_k$ since $b_i$ is unchanged for $i \leq \ell$, and
    \item $\{[b_{\ell+1}], \ldots, [b_{\lambda-1}], [\wh{b}_\lambda], [b_{\lambda+1}], \ldots, [b_d]\}$ is a basis for $(M^k/M^{k-1})_k$ since $[b_\lambda] + [\wh{b}_\lambda]$ is in \\ $M(\lambda-1)/M^{k-1}$ which is spanned by  $\CB \cup \{[b_1], \ldots, [b_{\lambda-1}]\}$. \qedhere
\end{itemize}
\end{proof}
Now that we've shown we can assume $Sq^1 b_\lambda$ only hits the top classes of available seagulls in $M(\lambda-1)$, we can exhibit an isomorphism between $M(\lambda)$ and a flock of seagulls. 
\begin{lemma}\label{leminductive}
If, for some $\lambda \geq \ell+1$,  $M(\lambda-1)$ is isomorphic to a flock of seagulls, then $M(\lambda)$ is isomorphic to a flock of seagulls. 
\end{lemma}
\begin{proof}
  Recall from Remark~\ref{rmkquotientbasis} that \[M^k/M^{k-1} \cong \bigoplus\limits_{1 \leq i \leq d} \Sigma^k \Upsilon_1\{[b_i]\}.\]
So, for all $1 \leq j \leq d$, \[M(j)/M^{k-1} \cong \bigoplus_{1 \leq i \leq j} \Sigma^k \Upsilon_1 \{[b_i]\}, \]
and therefore,
\[ M(\lambda)/M(\lambda-1) \cong \big(M(\lambda)/M^{k-1}\big)/\big(M(\lambda-1)/M^{k-1}\big) = \Sigma^{k}\Upsilon_{1}\{[b_{\lambda}]\}. \]

We have from Lemma~\ref{lemb_available} that 
\[ Sq^1 b_\lambda = Sq^2 Sq^1 Sq^2 \sum_{p \in P} y_{p,k-4} \]
for some $P \subseteq A_{\lambda-1}$.
Let $p_\lambda \in P$ be such that $\alpha_{p_\lambda}$ is minimal in $\{\alpha_p : p \in P\}$.  (That is, no seagull whose top class is hit by $Sq^1 b_\lambda$ starts below the $p_\lambda^{\text{th}}$ seagull.) This selects the longest seagull indexed by $P$ (or perhaps one of multiple longest seagulls).

Let
\[ X = \left(\bigoplus_{\substack{i \in I_{\lambda-1}\\i \neq p_\lambda }} \Sigma^{\alpha_i} \Upsilon_{n_i} \right) \oplus \Sigma^{\alpha_{p_\lambda}} \Upsilon_{n_{p_\lambda}+1}. \]
I.e., $X$ is obtained from $M(\lambda-1)$ by attaching a 1-seagull to the top of the $p_\lambda^{\text{th}}$ seagull. Since the left summand of $X$ is a submodule of $M(\lambda-1)$, we will keep the notation that $\Sigma^{\alpha_i}\Upsilon_{n_i}$ in this summand is generated by $y_{i, \alpha_i}, y_{i, \alpha_i+4}, \ldots, y_{i, \alpha_i+4(n_i-1)}$.  The generators of $\Sigma^{\alpha_{p_\lambda}}\Upsilon_{n_{p_\lambda}+1}$ will be named 
\[ x_{\alpha_{p_\lambda}}, x_{\alpha_{p_\lambda}+4}, \ldots, x_{k}. \]

Let $\Phi:X \to M(\lambda)$ be the $\CA(1)$-module map determined by:
\begin{align*}
    \Phi(y_{i,j}) &= y_{i,j}
\\
    \Phi(x_j) &= \begin{cases} \sum\limits_{p \in P} y_{p,j} & j<k \\ b_\lambda & j=k. \end{cases}
\end{align*}
If the generator $y_{p,j}$ does not exist for some $p \in P$, omit the term from the sum for $\Phi(x_j)$.  
On the left summand of $X$, it is clear that there is a well defined $\CA(1)$-map determined by this mapping of the generators.  For the right summand, we need to make sure \[\Phi(Sq^{1}x_{j})=\begin{cases} \Phi(Sq^{2}Sq^{1}Sq^{2}x_{j-4}) & j > \alpha_{p_\lambda} \\ 0 & j=\alpha_{p_\lambda}.  \end{cases}\]
In the case that $j=k$,
\begin{align*}
    \Phi(Sq^{1}x_{k}) &= Sq^{1}b_{\lambda} = Sq^{2}Sq^{1}Sq^{2}\sum_{p \in P} y_{p, k-4} = \Phi(Sq^{2}Sq^{1}Sq^{2}x_{k-4}).
\end{align*}
If $j<k$,
\begin{align*}
    \Phi(Sq^{1}x_{j}) &= Sq^{1}\sum_{p \in P} y_{p, j} = Sq^{2}Sq^{1}Sq^{2}\sum_{p \in P} y_{p, j-4}
\end{align*}
again, omitting any $y_{p, j-4}$ that does not exist. 
If $j>\alpha_{p_\lambda}$, then this is equal to $\Phi(Sq^{2}Sq^{1}Sq^{2}x_{j-4})$.  If $j=\alpha_{p_\lambda}$ then, since $\alpha_{p_\lambda} \leq \alpha_p$ for all $p \in P$, $Sq^1 y_{p,\alpha_{p_\lambda}}=0$ for all $p \in P$ such that $y_{p, \alpha_{p_\lambda}}$ exists. Thus, $\Phi\left(Sq^1 x_{\alpha_{p_\lambda}}\right)=0$.

Then we have the following commutative diagram of $\CA(1)$ modules:
\[\begin{tikzcd}
0 \ar[r] & \bigoplus\limits_{i \in I_{\lambda-1}} \Sigma^{\alpha_{i}}\Upsilon_{n_{i}} \ar[r] \ar[d, "\cong" turn] & X \ar[r] \ar[d, "\Phi"] & \Sigma^{k} \Upsilon_{1}\{x_k\} \ar[r] \ar[d, "\cong" turn] & 0 \\
0 \ar[r] & M(\lambda-1) \ar[r]  & M(\lambda) \ar[r]  & \Sigma^{k}\Upsilon_{1}\{b_{\lambda}\} \ar[r]  & 0 
\end{tikzcd}\]
The rows are exact so, by the 5-lemma, $\Phi$ is an isomorphism. 
\end{proof}

Proposition~\ref{propfiniteclass} follows by induction. 
In the following section, we will need not just the fact that $M^k$ is a flock of seagulls, but also the following more specific facts about the structure of $M^k$.
\begin{notation}\label{notationdecomp}
Since $M^k=M(d)$,
\[ M^{k} \cong \left(\bigoplus_{i \in L_{d}} \Sigma^{\alpha_i} \Upsilon_{n_i} \right) \oplus \left(\bigoplus_{i \in A_{d}} \Sigma^{\alpha_i} \Upsilon_{n_i} \right)
\oplus \left(\bigoplus_{i \in U_{d}} \Sigma^{\alpha_i} \Upsilon_{n_i} \right). \]
Recall from Definition~\ref{defnLAU} that the lengthened seagulls (indexed by $L_d$) have generators in degrees $k-4$ and $k$, the available seagulls (indexed by $A_d$) have a generator in degree $k-4$ but no generator in degree $k$, and  the unavailable seagulls (indexed by $U_d$) do not have a generator in degree $k-4$.  We will use the notation $L:= L_d$, $A:= A_d$, and $U:=U_d$.
\end{notation}
\begin{cor}\label{cordecomp}
There are $d-\ell$ lengthened seagulls in $M^k$.
\end{cor}
\begin{proof} 
Since 
\[M(\ell) \cong M^{k-1} \oplus \bigoplus_{1 \leq i \leq \ell} \Sigma^k \Upsilon_1, \]
the only seagulls in $M(\ell)$ with the top class in degree $k+5$ are of length one. So, $|L_\ell|=0$.

In the proof of Lemma~\ref{leminductive}, we showed
\[ M(\lambda) \cong  \left(\bigoplus_{\substack{i \in I_{\lambda-1}\\i \neq p_\lambda }} \Sigma^{\alpha_i} \Upsilon_{n_i} \right) \oplus \Sigma^{\alpha_{p_\lambda}} \Upsilon_{n_{p_\lambda}+1}\]
for some special $p_\lambda \in A_{\lambda-1}$.
So, for all $\lambda$,
\begin{align*}
    |L_\lambda|&=|L_{\lambda-1}| +1 \\
    |A_\lambda|&=|A_{\lambda-1}| - 1 \\
    U_\lambda&=U_{\lambda-1},
\end{align*}
Then $|L_d|=d -\ell$.
\end{proof}

\subsection{Classification of Bounded Below Modules}\label{secinfinitesg}
In classifying $Q_0$-local modules that are bounded below and of finite type, but possibly infinite, our strategy is to decompose the module into a sum of finite modules and a summand with no finite summands.  The sum of finite modules is then already shown to be a flock of seagulls, and we only need to consider the term with no finite summands.  We begin by showing this type of decomposition exists.  In the following lemmas, $\CB$ is any finite subalgebra of the Steenrod algebra.

\begin{lemma}\label{lempartone}
For any bounded below $\CB$-module, $K$, of finite type, and any integer $n$, there is a decomposition $K \cong N[n] \oplus K(n)$, such that $N[n] \subseteq K^n$ and $K(n)$ has no direct summands that are submodules of $K^n$.
\end{lemma} 
\begin{proof}
If no nonzero submodule of $K^n$ splits off of $K$, let $N[n]=0$ and $K(n)=K$.  

Otherwise, let $N^{(1)}$ be any nonzero submodule (possibly zero) of $(K^{n})^{(1)}:=K^n$ that splits off of $K$.  
Given a decomposition of $K$,
 \[ K \cong N^{(1)} \oplus N^{(2)} \oplus \cdots \oplus N^{(k-1)} \oplus K', \]
 where each $N^{(i)}$ is a submodule of $K^n$, let $N[n]=  N^{(1)} \oplus N^{(2)} \oplus \cdots \oplus N^{(k-1)}$ and $K(n)=K'$ if $K'\cap K^n$ has no nonzero submodule that splits off of $K$.  Otherwise, let $N^{(k)}$ be any nonzero submodule of $K' \cap K^n$ that splits off of $K$.  
Since $K^n$ is finite, iteration of this process will conclude with the selection of an $N[n]$ and $K(n)$.
\end{proof}

\begin{lemma}\label{lemparttwo}
Any bounded below $\CB$-module, $M$, of finite type, splits as a direct sum, \[M(\infty) \oplus N(\infty),\] where $N(\infty)$ is a (possibly infinite) direct sum of finite modules and $M(\infty)$ has no finite summands. 
\end{lemma}

\begin{proof}  After possibly shifting $M$, we may assume $M_k=0$ for all $k\leq 0$. 
Applying Lemma~\ref{lempartone} to $K=M$ with $n=0$, we obtain $N[0]$ and $M(0)$ such that $M \cong N[0] \oplus M(0)$, $N[0] \subseteq M^0$, and $M(0)$ has no direct summands contained in $M^0$.
Given \[M \cong N[0] \oplus \cdots \oplus N[n-1] \oplus M(n-1)\] where $N[k] \subseteq M^{k}$ for all $k$ and $M(n-1)$ has no direct summands contained in $M^{n-1}$, we apply Lemma~\ref{lempartone} to $K=M(n-1)$ and obtain modules $N[n]$ and $M(n)$ such that
\[ M(n-1) \cong N[n] \oplus M(n), \]
$N[n] \subseteq M(n-1)^n \subseteq M^n$, and $M(n)$ has no direct summands contained in $M(n-1)^n$ (and thus no direct summands contained in $M^n$). 

We take $N(\infty)$ to be $\left( \bigoplus\limits_{n \geq 0} N[n]\right)$. Each $N[n]$ is finite, as it is a submodule of the finite module $M^n$, so this satisfies the description of $N(\infty)$ as a direct sum of finite modules. 
 The intersection of the $M(n)$'s has no finite summand, since any finite summand would split off of some $M^n$ and thus would not be contained in $M(n)$. So, we can take $M(\infty)$ to be $\left(\bigcap\limits_{n \geq 0} M(n) \right)$.
\end{proof}

So, in order to prove Theorem~\ref{thmclassification}, it remains to classify reduced, bounded below $Q_0$-local $\CA(1)$-modules of finite type that have no finite summands.  We will use the following lemma to help in identifying finite summands.   
\begin{lemma}\label{lemsplit}
Let $M$ be a bounded below $\CA(1)$-module.  If $N$ is a direct summand of $M^{k}$ and all elements of $N$ have degree no more than $k+1$, then $N$ splits off as a direct summand of $M$.
\end{lemma}
\begin{proof}
The inclusion of $N$ into $M^k$ is split by the projection $\pi:M^k \to N$.  Let $p:M \to N$ be the map of $\FF_2$-vector spaces given by 
\[ p(x) = \begin{cases} \pi(x) & x \in M^k \\ 0 & \text{otherwise.} \end{cases} \]
For any $x \in M^k$ and any $R$, $Sq^R x \in M^k$, so
\[ p\left(Sq^R x\right) = \pi\left( Sq^R x \right) = Sq^R \pi(x) = Sq^R p(x). \]

If $x \notin M^k$, then $|x| \geq k+1$. So, for any $R\neq (0)$, $|Sq^R x| \geq k+2$.  So, if $Sq^R x \in M^k$ then $p\left( Sq^R x \right)=\pi\left(Sq^R x \right)$ is zero, as the top class of $N$ is in degree $k+1$.  If $Sq^R x \notin M^k$ then $p\left( Sq^R x\right)$ is defined to be zero.  
Hence, for any $x \notin M^k$, 
\[ p\left( Sq^R x\right) = 0 = Sq^R (0) = Sq^R p(x). \]
So, $p$ is an $\CA(1)$-map. 

Let $i:N \to M$ be the inclusion.  The image of $i$ is contained in $M^k$, so for all $x \in N$,
\[ p \circ i (x) = \pi\big(i(x)\big)=x. \]
This exhibits a splitting of $N$ off of $M$ as a direct summand. 
\end{proof} 

\begin{defn}
For any reduced, bounded below, $Q_0$-local $\CA(1)$-module of finite type, $M$, and any $k \geq 0$, let \[V_{k}=\FF_{2}\{b_{\ell+1}, \ldots, b_{d}\},\] where the $b_{i}$ are chosen from $M^{k}$ as described in Definition~\ref{defnb} and Lemma~\ref{lemchoose_b}.  Also, let \[W_k=\im(Sq^2 Sq^1 Sq^2)_{k+1}\] 
in $M$. 
\end{defn}
For any nonzero class $Sq^2 Sq^1 Sq^2 x \in W_k$, the degree of $x$ is $k-4$, so $x \in M^{k-1}$ and $Sq^2 Sq^1 Sq^2 x \in M^{k-1}$. 
We also note that, by construction of the $b_i$'s, 
\[ \mathrm{dim}(V_k) = \mathrm{dim}\left(\left(M^k/M^{k-1} \right)_k\right) - \mathrm{dim}\left(\left(K/M^{k-1}\right)_k \right) \]
where $K$ is the kernel of the $Sq^1$ action.

\begin{lemma}\label{lemdim}
If $M$ is a reduced, bounded below, $Q_{0}$-local $\CA(1)$-module of finite type and there exists some $k$ such that $\dim(W_{k}) \neq \dim(V_{k})$, then $M$ has a direct summand $\Sigma^\alpha \Upsilon_{n}$ for some finite $n$ and some $\alpha$.
\end{lemma}
\begin{proof}
As in Notation~\ref{notationdecomp}, we take the decomposition of $M^k$, 
\[ M^{k} \cong  \left(\bigoplus_{\substack{i \in U}} \Sigma^{\alpha_i}\Upsilon_{n_i} \right) \oplus \left(\bigoplus_{\substack{i \in A}} \Sigma^{\alpha_i}\Upsilon_{n_i} \right)  \oplus  \left(\bigoplus_{i \in L} \Sigma^{\alpha_{i}}\Upsilon_{n_{i}} \right)  \]
recalling that $A \cup L$ indexes all seagulls in $M^k$ with a generator in degree $k-4$. In Corollary~\ref{cordecomp}, we showed $|L|=d-\ell$. Note that $d-\ell$ is also equal to the dimension of $V_k$.

Since $M^{k-1}$ is a flock of seagulls, the map
\[Sq^{2}Sq^{1}Sq^{2}: \FF_{2}\{y_{i, k-4} \in \big(M^{k-1}\big)_{k-4} \} \to \im(Sq^{2}Sq^{1}Sq^{2})_{k+1} =W_k\]
is an isomorphism of $\FF_2$-vector spaces. Thus,
\begin{align*} \mathrm{dim}\left( W_k \right)&= \mathrm{dim}\left(\FF_2\left\{ y_{i,k-4} \in \left(M^{k-1}\right)_{k-4}\right\} \right)=|L|+|A|  = (d-\ell)+|A| = \mathrm{dim}\left(V_k\right) + |A|. \end{align*}
So, if $\mathrm{dim}\left(W_k\right) \neq \mathrm{dim}\left(V_k\right)$, then $|A|>0$. Fixing some $r \in A$ gives a direct summand of $M^{k}$, $\Sigma^{\alpha_{r}}\Upsilon_{n_{r}}$, whose top class is in degree $k+1$. By Lemma~\ref{lemsplit}, this finite seagull splits off as a direct summand of $M$.
\end{proof}

\begin{lemma}
If $M$ is a reduced, bounded below, $Q_0$-local $\CA(1)$-module of finite type with no finite summands, then there exists a family of isomorphisms, 
\[ \Phi^k:\bigoplus_{i \in I_{j}} \Sigma^{\beta_i}\Upsilon_{m_i^j} \to M^j \]
satisfying the following conditions:
\begin{enumerate}[label=\roman*.)]
    \item \textbf{Inclusion:} Every seagull in $M^{j-1}$ corresponds to a seagull in $M^j$ and any new seagull in $M^j$ is a 1-seagull starting in degree $j$. I.e.  $I_{j-1} \subseteq I_{j}$, and for any $i \in I_{j} \setminus I_{j-1}$, $\beta_i=j$ and $m_i^j=1$,
    
    \item \textbf{Maximum length:} A seagull in $M^j$ corresponding to a $i^{\text{th}}$-seagull in $M^{j-1}$ has an additional 1-seagull attached to the top whenever allowed by the degree of the seagull's top class.  I.e., for all $i \in I_{j-1}$, $m_i^j = \begin{cases} m_i^{j-1}+1 & \beta_i \equiv j \pmod{4} \\
m_i^{j-1} & \text{otherwise,} \end{cases}$ and consequently $m_i^j=\left\lfloor \dfrac{j-\beta_i}{4} \right\rfloor$.
\item \textbf{Decomposition:} The isomorphisms respect the direct sums. 
I.e., for each $j\geq 1$, the following diagram commutes. 
\[\begin{tikzcd}   
\bigoplus\limits_{i \in I_{j-1}}  \Sigma^{\beta_i} \Upsilon_{m_i^{j-1}} \ar[d, "\Phi^{j-1}"] \ar[r, "f^j"] & \bigoplus\limits_{i\in I_{j} } \Sigma^{\beta_i} \Upsilon_{m_i^j} \ar[d, "\Phi^j"] \\
M^{j-1} \ar[r, "\iota^j"] & M^j
\end{tikzcd} \]
where $f^j$ and $\iota^j$ are the obvious inclusions. In particular, $f^j$ maps $\Sigma^{\beta_i}\Upsilon_{m_i^{j-1}}$ into $\Sigma^{\beta_i}\Upsilon_{m_i^j}$.
\end{enumerate}
\end{lemma}
\begin{figure}[h]
    \centering
    \caption{An example of a family of seagulls indexed by sets satisfying inclusion and maximum length. }
    \label{fig:Ij}
    \hfill \begin{tikzpicture}[scale=0.3]
     \draw[gray, dashed] (-1,10)--(9,10);
     \node[left] at (14,10){\text{deg. }$j$};
     \SgL(0,0);
     \SgR(0,4);
     \sqone(0,4,);
     \SgL(0,8);
     \sqone(0,8,);
     \SgL(2,2);
     \SgR(2,6);
     \sqone(2,6,);
     \SgL(4,3);
     \SgR(4,7);
     \sqone(4,7,);
     \SgL(6,6);
     \SgL(8,9);
     \node[below] at (4,0){ $\bigoplus\limits_{i \in I_{j-1}} \Sigma^{\beta_i} \Upsilon_{m_i^{j-1}}$};
    \end{tikzpicture} \hfill  \begin{tikzpicture}[scale=0.3]
     \draw[gray, dashed] (-1,10)--(11,10);
     \node[left] at (16,10){\text{deg. }$j$};
     \SgL(0,0);
     \SgR(0,4);
     \sqone(0,4,);
     \SgL(0,8);
     \sqone(0,8,);
     \SgL(2,2);
     \SgR(2,6);
     \sqone(2,6,);
     \SgL(2,10);
     \sqone(2,10,);
     \SgL(4,3);
     \SgR(4,7);
     \sqone(4,7,);
     \SgL(6,6);
     \SgR(6,10);
     \sqone(6,10,);
     \SgL(8,9);
     \SgL(10,10);
     \node[below] at (4,0){ $\bigoplus\limits_{i \in I_{j}} \Sigma^{\beta_i} \Upsilon_{m_i^{j}}$};
    \end{tikzpicture} \hfill \phantom{.} 
\end{figure}
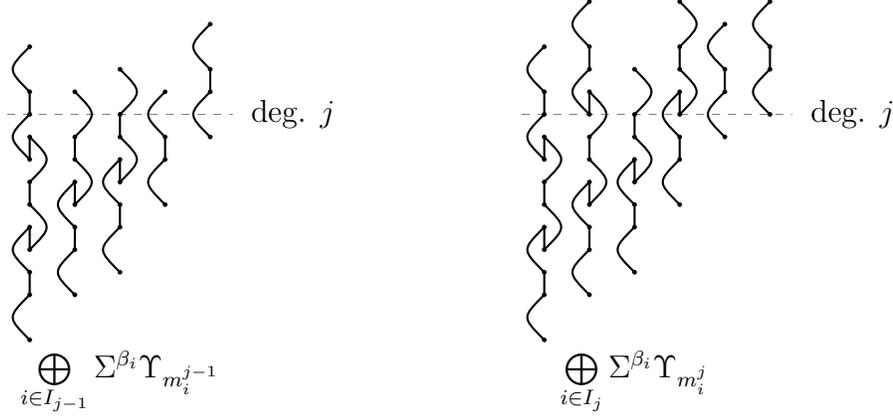
\begin{proof}
We assume that $M_k=0$ for all $k<0$, after a possible shift.  For all $k$, $M^k$ is finite and thus isomorphic to a flock of seagulls (Proposition~\ref{propfiniteclass}).   In the case where $k=0$, we fix any isomorphism \[\Phi^0: \bigoplus\limits_{i \in I_0} \Sigma^{0}\Upsilon_{1} \to M^0.\]
Now suppose that for each $j \leq k-1$, there is a flock of seagulls, $\bigoplus\limits_{i \in I_j} \Sigma^{\beta_i}\Upsilon_{m_i^j}$, mapping into $M^j$ via a fixed isomorphism $\Phi^j$ such that $\{\Phi^j\}_{j \leq k-1}$ satisfies the three desired conditions. We will then produce an isomorphism $\Phi^k$ so that the family $\{\Phi^j\}_{j \leq k}$ satisfies those conditions as well.

We have the decomposition of $M^k$ into unavailable, available, and lengthened seagulls,
\[ M^{k} \cong  \left(\bigoplus_{\substack{i \in U}} \Sigma^{\alpha_i}\Upsilon_{n_i} \right) \oplus \left(\bigoplus_{\substack{i \in A}} \Sigma^{\alpha_i}\Upsilon_{n_i} \right)  \oplus  \left(\bigoplus_{i \in L} \Sigma^{\alpha_{i}}\Upsilon_{n_{i}} \right),  \]
where $L \cup A$ indexes all seagulls in $M^{k}$ with a generator in degree $k-4$. \FloatBarrier
In Lemma~\ref{lemdim}, we showed that in the case where $M$ has no finite summands,\\ $\mathrm{dim}\left(W_k\right)=\mathrm{dim}\left(V_k \right)$ and so $|A|=0$. So, in this case, 
\begin{align*} M^k &\cong \left(\bigoplus_{\substack{i \in U}} \Sigma^{\alpha_i}\Upsilon_{n_i} \right)  \oplus  \left(\bigoplus_{i \in L} \Sigma^{\alpha_{i}}\Upsilon_{n_{i}+1} \right). 
\end{align*}

\begin{figure}[ht!]
    \centering
    \caption{An example of $M^k$ when $M$ has no finite summands.}
    \label{fig:infinite}
    \begin{tikzpicture}[scale=0.4]
     \draw[gray, dashed] (-9,10)--(14,10);
     \node[left] at (17,10){\text{deg. }$k$};
        \SgL(11,0);
        \SgR(11,4);
        \sqone(11,4,);
        \SgL(11,8);
        \sqone(11,8,);
        \SgR(13,7);
        \SgL(9,9);
        \SgR(0,10);
        \SgR(2,10);
        \SgL(-7,10);
        \SgR(-7,6);
        \sqone(-7,6,);
        \SgL(-7,2);
        \sqone(-7,10,);
        \SgR(-5,10);
        \SgL(-5,6);
        \sqone(-5,10,);
        \draw[gray](-9,3)--(-9,1)--(-4,1)--(-4,7);
        \draw[gray](-6.5,1)--(-6.5,0);
         \node[below] at (-6.5,0){\Large $\bigoplus\limits_{i \in L} \Sigma^{\alpha_i}\Upsilon_{n_i}$};
        \node[below] at (-6.5,-3){Lengthened};
        \draw[gray](-1,11)--(-1,9)--(4,9)--(4,11);
        \draw[gray](1.5,9)--(1.5,8);
        \node[below] at (1.5, 8){\Large $\bigoplus\limits_{\substack{i \in U \\ \alpha_i = k}} \Sigma^{\alpha_i}\Upsilon_{n_i}$};
        \node[below] at (1.5,4){Unavailable with a};
        \node[below] at (1.5,3){generator in degree $k$};
        \draw[gray](7,9.75)--(7,-1)--(15,-1)--(15,8);
        \draw[gray](10,-1)--(10,-2);
        \node[below] at (10,-2){\Large $\bigoplus\limits_{\substack{i \in U \\ \alpha_i \neq k}} \Sigma^{\alpha_i}\Upsilon_{n_i}$};
        \node[below] at (10,-6){Unavailable, no generator in degree $k$};
    \end{tikzpicture}
\end{figure}
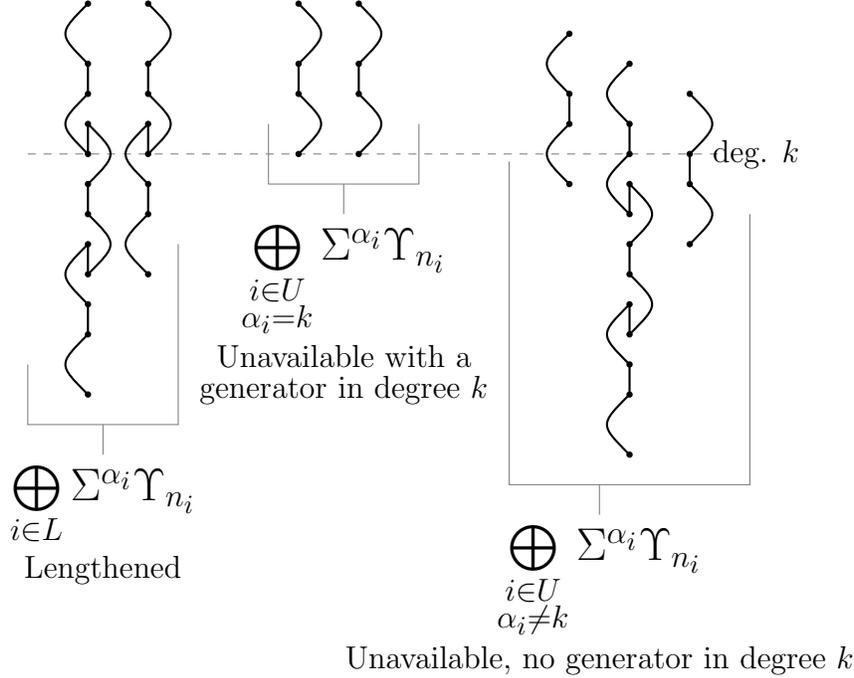
\FloatBarrier
Suppose $i \in U$ with $\alpha_i \equiv k \pmod{4}$.
   By Lemma~\ref{lemsplit}, the top generator of $\Sigma^{\alpha_i}\Upsilon_{n_i}$ must be in degree at least $k-3$. Otherwise, this finite seagull would split off as a summand of $M$, which we have assumed has no finite summands.   The degree of the top generator is also no more than $k$ and equivalent to $k$ modulo 4, so the top generator is in degree $k$. Since this seagull is not lengthened, the generator in degree $k$ is the only generator.  Thus, $\alpha_i=k$ and $n_i=1$.
Thus, 
\begin{align*} M^k 
&\cong \left(\bigoplus_{\substack{i \in U \\ \alpha_i = k }} \Sigma^{\alpha_i}\Upsilon_{1} \right) \oplus  \left(\bigoplus_{\substack{i \in U \\ \alpha_i\not\equiv  k \ (4) }} \Sigma^{\alpha_i}\Upsilon_{n_i} \right)   \oplus  \left(\bigoplus_{i \in L} \Sigma^{\alpha_{i}}\Upsilon_{n_{i}} \right).
\end{align*} 
For any $i \in U$ with $\alpha_i \not\equiv k \pmod{4}$, $\Sigma^{\alpha_i}\Upsilon_{n_i}$ cannot have a generator in degree $k$.  (All generators of a seagull have equivalent degrees modulo 4.)  Since $\Sigma^{\alpha_i}\Upsilon_{n_i} \subseteq M^k$, this means all generators of this seagull have degree no more than $k-1$.
So,
\[ M^k/M^{k-1} \cong  \left(\bigoplus_{\substack{i \in U \\ \alpha_i=k }} \Sigma^{k}\Upsilon_{1} \right) \oplus \left(\bigoplus_{i \in L} \Sigma^{k}\Upsilon_{1} \right). \]
We wish to construct
\[ \bigoplus_{i \in I_k} \Sigma^{\beta_i} \Upsilon_{m_i^k} \]
which will be isomorphic to $M^k$.  To do so, we will lengthen some seagulls in 
\[ \bigoplus_{i \in I_{k-1}} \Sigma^{\beta_i} \Upsilon_{m_i^{k-1}} \]
and introduce some new 1-seagulls. 
Let $I_k= I_{k-1}\sqcup \{ i \in U| \alpha_i= k \}$, setting $\beta_{i}=k$ for all $i \in U$ such that $\alpha_i=k$, and leaving $\beta_i$ unchanged for all $i \in I_{k-1}$.  Then, for all $i \in I_k$, let
\begin{align*} m_i^k &= \begin{cases} m_i^{k-1} & i \in I_{k-1}, \ \beta_i \not\equiv  k \pmod{4}\\ m_i^{k-1}+1 & i \in I_{k-1}, \ \beta_i \equiv k \pmod{4} \\ 
1 & i \in \{q \in U| \alpha_q =k\}. 
\end{cases}   \end{align*}
Then the inclusion and maximum length properties are satisfied by $\{I_j\}_{j \leq k}$.  Consider the decomposition
\begin{align} \label{eqdecomp} \bigoplus_{i \in I_k} \Sigma^{\beta_i}\Upsilon_{m_i^k} \cong \left(\bigoplus_{\substack{i \in I_{k-1}\\ \beta_i\not\equiv k \ (4)}} \Sigma^{\beta_i} \Upsilon_{m_i^{k-1}}\right) \oplus \left(\bigoplus_{\substack{i \in I_{k-1}\\ \beta_i\equiv k \ (4) }} \Sigma^{\beta_i} \Upsilon_{m_i^{k-1}+1}\right) \oplus \left(\bigoplus_{\substack{i \in I_{k}\setminus I_{k-1} }} \Sigma^{k} \Upsilon_{1}\right).  \end{align}
We will construct a map, \[ \Phi^k: \bigoplus_{i \in I_k} \Sigma^{\beta_i}\Upsilon_{m_i^k} \to M^k, \] by defining $\Phi^k$
on each summand of \eqref{eqdecomp}.

The leftmost summand is contained in $\bigoplus\limits_{i \in I_{k-1}} \Sigma^{\beta_i}\Upsilon_{m_i^{k-1}}$, so we define $\Phi^k$ on this summand to be the composition,
\[ \bigoplus_{\substack{i \in I_{k-1}\\ \beta_i \not\equiv k \ (4) }} \Sigma^{\beta_i} \Upsilon_{m_i^{k-1}} \xrightarrow{\Phi^{k-1}} M^{k-1} \xrightarrow{\iota^k} M^k. \]
By the definition of $I_k$, we have
\[ \bigoplus_{i \in I_k \setminus I_{k-1}} \Sigma^k \Upsilon_1 = \bigoplus_{\substack{i \in U\\ \alpha_i=k}} \Sigma^k \Upsilon_1 \subseteq M^k, \]
so we define $\Phi^k$ on the rightmost summand of \eqref{eqdecomp} to be the inclusion into $M^k$.
Defining the map on the center summand is more complicated.  Recall from Lemma~\ref{lemchoose_b} that
\[ Sq^1 b_i \in Sq^2 Sq^1 Sq^2 \left( \left(M^{k-1}\right)_{k-4} \right)=W_k \]
for all $\ell+1 \leq i \leq d$.
So, $U_k:=\FF_2 \left\{ Sq^1 b_i | \ell+1 \leq i \leq d \right\} \subseteq W_k$.  However,
\[ Sq^1: V_k \to U_k \]
is a vector space isomorphism, since no combination of the $b_i$'s are in the kernel of $Sq^1$.  
By Lemma~\ref{lemdim}, $\mathrm{dim}\left(U_k\right) = \mathrm{dim}\left(V_k\right) = \mathrm{dim}\left(W_k \right)$. So, $U_k=W_k$.  

Let $\Sigma^{\beta_i}\Upsilon_{m_i^k}=\Sigma^{\beta_i}\Upsilon_{m_i^{k-1}+1}$ in the center summand of~\eqref{eqdecomp} have generators 
\[ z_{i, \beta_i}, z_{i, \beta_i+4}, \ldots, z_{i, k}. \]
For any $z_{i,j}$ with $j<k$, $z_{i,j}$ is in $M^{k-1}$, so $\Phi^{k-1}(z_{i,j})$ is defined. 
When $j=k-4$, \[\Phi^{k-1}\left(Sq^2 Sq^1 Sq^2 z_{i,k-4}\right) \in Sq^2 Sq^1 Sq^2 \left( \left(M^{k-1}\right)_{k-4} \right)=W_k=U_k.\]
So, $\Phi^{k-1}\left(Sq^2 Sq^1 Sq^2 z_{i, k-4}\right)=\sum\limits_{\ell+1 \leq q \leq d} Sq^1 c_i^q b_q$, where $c_i^q \in \FF_2$ for all $i$ and $q$.  Define $\Phi^k$ on this summand to be the $\CA(1)$-map determined by
\[ \Phi^k(z_{i,j}) = \begin{cases} \iota^k \circ \Phi^{k-1}(z_{i,j}) & j<k \\
\sum\limits_{\ell+1 \leq q \leq d} c_i^q b_q & j=k.\end{cases} \]
To see such a map is well defined, we need to check 
\[ \Phi^k\left(Sq^2 Sq^1 Sq^2 z_{i,k-4}\right) = \Phi^k\left(Sq^1 z_{i,k-4}\right) \]
for all $i$, but $\Phi^k$ was constructed so that this relation holds. Furthermore, the diagram
\[\begin{tikzcd}   
\bigoplus\limits_{i \in I_{k-1}}  \Sigma^{\beta_i} \Upsilon_{m_i^{k-1}} \ar[d, "\Phi^{k-1}"] \ar[r, "f^k"] & \bigoplus\limits_{i\in I_{k} } \Sigma^{\beta_i} \Upsilon_{m_i^k} \ar[d, "\Phi^k"] \\
M^{k-1} \ar[r, "\iota^k"] & M^k
\end{tikzcd} \]
commutes by the construction of $\Phi^k$, so the family $\left\{ \Phi^j \right\}_{j \leq k}$ satisfies the decomposability property.

To show $\Phi^{k}$ is a isomorphism, we will apply the five lemma to the diagram,
\[
\begin{tikzcd}
0 \ar[r] & \bigoplus\limits_{i \in I_{k-1}} \Sigma^{\beta_i} \Upsilon_{m_i^{k-1}} \ar[r, "f^k"] \ar[d, "\Phi^{k-1}"]  & \bigoplus\limits_{i \in I_{k}} \Sigma^{\beta_i} \Upsilon_{m_i^{k}} \ar[r, "\hat{q}"] \ar[d, "\Phi^k"] &  \bigoplus\limits_{\mathclap{\substack{i \in I_{k-1}\\ \beta_i \equiv k \ (4)}}}\Sigma^{k}\Upsilon_1 \{[z_{i, k}]\} \oplus \bigoplus\limits_{\mathclap{\substack{i \in U\\ \alpha_i=k}}} \Sigma^k \Upsilon_1 \ar[r] \ar[d, "\Lambda"] & 0 \\
0 \ar[r] & M^{k-1} \ar[r, "\iota^k"] & M^{k}\ar[r, "q"] &  \bigoplus\limits_{i \in L}\Sigma^{k}\Upsilon_1 \oplus \bigoplus\limits_{\mathclap{\substack{i \in U\\ \alpha_i=k}}} \Sigma^k \Upsilon_1 \ar[r]  & 0 \\
\end{tikzcd}
\]
where $\Lambda$ takes $[z_{i, k}]$ to $\sum\limits_{\ell+1 \leq q \leq d} c_i^q [b_q]$ and is the identity on the right summand.  We have already seen that the left square commutes by the definition of $\Phi^k$.  To see that the right square commutes, we consider the decomposition given in \eqref{eqdecomp}.   On the summand contained in $\bigoplus\limits_{i \in I_{k-1}} \Sigma^{\beta_i}\Upsilon_{m_i^{k-1}}$,
\[ q \circ \Phi^k = 0 = \Lambda \circ \hat{q}. \]
On the summand isomorphic to $\bigoplus\limits_{\mathclap{\substack{i \in U\\ \alpha_i=k}}} \Sigma^{k} \Upsilon_{1}$, both compositions are the identity. On the remaining summand,
\begin{align*}
    q \circ \Phi^k (z_{i,j}) &= \begin{cases} \left[\sum\limits_{\ell+1 \leq q \leq d} c_i^q b_q \right] & j = k \\ 0 & \text{otherwise} \end{cases} \\
    &= \begin{cases} \Lambda([z_{i,k}]) & j = k \\ \Lambda(0) & \text{otherwise} \end{cases} \\
    &= \Lambda \circ \hat{q} (z_{i,j}).
\end{align*}

Finally, we show $\Lambda$ is an isomorphism, completing the proof. Certainly, the restriction of $\Lambda$ to the summand $\displaystyle\bigoplus_{\substack{i \in U \\ \alpha_i \equiv k \ (4)}} \Sigma^k \Upsilon_1$ is an isomorphism.    By the definition of $\Lambda$, the image of this summand under $\Lambda$ is contained in $\displaystyle\bigoplus_{i \in L} \Sigma^k \Upsilon_1$.  To show the restriction of  $\Lambda$ to the summand $\displaystyle \bigoplus_{\substack{i \in I_{k-1} \\ \beta_i \equiv k \ (4)}} \Sigma^k \Upsilon_1\{[z_{i,k}]\}$ is an isomorphism, we show \[\Lambda_k: \left( \bigoplus_{\substack{i \in I_{k-1} \\ \beta_i \equiv k \ (4)}} \Sigma^k \Upsilon_1\{[z_{i,k}]\} \right)_k \to \left( \bigoplus_{i \in L} \Sigma^k \Upsilon_1\right)_k\] is an isomorphism of $\FF_2$-vector spaces.

By construction, $\{[b_{\ell+1}], \ldots, [b_d]\}$ is a basis for the $\FF_2$-vector space, $\displaystyle\left( \bigoplus_{i \in L} \Sigma^k \Upsilon_1 \right)_k$.  So, this vector space has dimension $d-\ell$.  The $\FF_2$-vector space, \[\displaystyle \left(\bigoplus_{\substack{i \in I_{k-1}\\ \beta_i \equiv k \ (4)}} \Sigma^k \Upsilon_1 \{[z_{i,k}]\}\right)_k\] is isomorphic (via the map given by the action of $Sq^1$) to 
\begin{align*}Sq^2 Sq^1 Sq^2 \left( \left(\bigoplus_{i \in I_{k}} \Sigma^{\beta_i}\Upsilon_{m_i^{k}} \right)_{k-4} \right) &\cong Sq^2 Sq^1 Sq^2 \left( \left(\bigoplus_{i \in I_{k-1}} \Sigma^{\beta_i}\Upsilon_{m_i^{k-1}} \right)_{k-4} \right)\\
&\cong Sq^2 Sq^1 Sq^2 \left(\left(M^{k-1}\right)_{k-4}\right)\\
&=W_k \end{align*}
The dimension of $W_k$ is also $d-\ell$.

Finally, we show $\Lambda_k$ is injective.  Take some element of the kernel of $\Lambda_k$,  $\sum\limits_{\substack{i \in I_{k-1}\\ \beta_i \equiv k \ (4)}} a_i [z_{i,k}]$  where each $a_i$ is in $\FF_2$.  Then, 
\begin{align*}
    0 &= \Lambda_k \left(\sum_{\substack{i \in I_{k-1}\\ \beta_i \equiv k \ (4)}} a_i [z_{i,k}] \right) = \sum_{\substack{i \in I_{k-1}\\ \beta_i \equiv k \ (4)}} a_i \left(\sum_{\ell+1 \leq q \leq d} c_i^q [b_q] \right).\\
\end{align*}
Since $Sq^1: \left(\bigoplus_{i \in L} \Sigma^k \Upsilon_1 \right)_k \to \left(\bigoplus_{i \in L} \Sigma^k \Upsilon_1 \right)_{k+1}$ is an isomorphism, 
\begin{align*}
    0 &= \sum_{\substack{i \in I_{k-1}\\ \beta_i \equiv k \ (4)}} a_i \left(\sum_{\ell+1 \leq q \leq d} Sq^1 c_i^q b_q\right) \\
    &= \sum_{\substack{i \in I_{k-1}\\ \beta_i \equiv k \ (4)}} a_i \Phi^{k-1} \left(Sq^2 Sq^1 Sq^2 z_{i,k-4} \right) \\
    &= \Phi^{k-1}\left(Sq^2 Sq^1 Sq^2 \sum_{\substack{i \in I_{k-1} \\ \beta_i \equiv k \ (4)}} a_i z_{i, k-4} \right).
\end{align*}
The maps $\Phi^{k-1}$ and \[\displaystyle Sq^2 Sq^1 Sq^2: \left( \bigoplus_{\substack{i \in I_{k-1}\\ \beta_i \equiv k \ (4)}} \Sigma^{\beta_i} \Upsilon_{m_i^k} \right)_{k-4} \to \left( \bigoplus_{\substack{i \in I_{k-1}\\ \beta_i \equiv k \ (4)}} \Sigma^{\beta_i} \Upsilon_{m_i^k} \right)_{k-4} \] are isomorphisms, so we can conclude
\[ 0= \sum_{\substack{i \in I_{k-1} \\ \beta_i \equiv k \ (4)}} a_i z_{i, k-4}. \qedhere \]

\end{proof}

We are now ready to prove the following lemma, which completes the proof of Theorem~\ref{thmclassification}. 
\begin{lemma}
If $M$ is a reduced, bounded below, $Q_0$-local $\CA(1)$-module of finite type with no finite summands, $M$ is isomorphic to a sum of suspensions of $\Upsilon_\infty$.
\end{lemma}
\begin{proof}
By the previous lemma, we can assume we have the following diagram for all $j \geq 1$,
\[\begin{tikzcd}   
\cdots \ar[r] &\bigoplus\limits_{i \in I_{j-1}}  \Sigma^{\beta_i} \Upsilon_{m_i^{j-1}} \ar[d, "\Phi^{j-1}"] \ar[r, "f^j"] & \bigoplus\limits_{i\in I_{j} } \Sigma^{\beta_i} \Upsilon_{m_i^j} \ar[d, "\Phi^j"] \ar[r] & \cdots \\
\cdots \ar[r] & M^{j-1} \ar[r, "\iota^j"] & M^j \ar[r] & \cdots
\end{tikzcd} \]
where $\Phi^j$ is an isomorphism for all $j$.
Then, 
\[ M \cong \varinjlim_{j} M^j \cong \varinjlim_{j} \left(\bigoplus\limits_{i\in I_{j} } \Sigma^{\beta_i} \Upsilon_{m_i^j}\right)  \]
along the maps in the diagram.  For ease of notation, we define $\Sigma^{\beta_i}\Upsilon_{m_i^j}$ to be zero for all $i \notin I_j$.  Since $f^j$ maps $\Sigma^{\beta_i}\Upsilon_{m_i^{j-1}}$ into $\Sigma^{\beta_i} \Upsilon_{m_i^j}$, 
\[\varinjlim_{j} \left(\bigoplus\limits_{i\in {I_{j}} } \Sigma^{\beta_i} \Upsilon_{m_i^j}\right) \cong \varinjlim_{j} \left(\bigoplus\limits_{ i \in \cup I_j } \Sigma^{\beta_i} \Upsilon_{m_i^j}\right) \cong \bigoplus\limits_{i \in {\cup I_j}}\left( \varinjlim_{j} \Sigma^{\beta_j}\Upsilon_{m_i^j}\right). \] 

Since our family of isomorphisms has the maximum length property,
\[ \Sigma^{\beta_i} \Upsilon_{m_i^j} = \Sigma^{\beta_i}\Upsilon_{p(i,j)} \]
where
\[ p(i,j) := \left\lfloor \frac{j-\beta_i}{4} \right\rfloor
\]
for all $j \geq \beta_i$. So, 
\[ \varinjlim_{j} \Sigma^{\beta_i}\Upsilon_{m_i^j} \cong \varinjlim_{j} \Sigma^{\beta_i}\Upsilon_{p(i,j)} \cong \varinjlim_{n} \Sigma^{\beta_i}\Upsilon_{n}  \]
where the maps in the last colimit are the inclusions $\Sigma^{\beta_i}\Upsilon_n \to \Sigma^{\beta_i}\Upsilon_{n+1}$.
Then we have 
\[ M \cong \bigoplus_{i \in \cup I_j}\left( \varinjlim_n \Sigma^{\beta_i}\Upsilon_n\right) = \bigoplus_{i \in \cup I_j} \Sigma^{\beta_i} \Upsilon_\infty.  \qedhere \]
\end{proof}

So, under appropriate finiteness conditions, every $Q_0$-local $\CA(1)$-module is stably equivalent to a flock of seagulls.  A quick consequence of this is that the $Q_0$-homology of a finite, $Q_0$-local $\CA(1)$-module must be of even dimension with generators forming pairs with difference in degree $4n+1$ for some positive integer $n$.  Since the $Q_0$-homology of $\Upsilon_\infty$ is one dimensional, no such condition applies in the case where $M$ is infinite. 

If $M$ is finite but not $Q_0$-local, we might like to describe the possibilities for
\[ H_\bullet (M;Q_0) \cong H_\bullet (L_0 M; Q_0) \]
in the same way. 
This is not possible, as $L_0 M$ is always infinite (though sometimes stably equivalent to a finite module).  However, in the following section, we are able to use the classification theorem in conjunction with $Q_0$-localization to give some results that apply to bounded below $\CA(1)$-modules of finite type with any $Q_1$-homology. 

\section{Applications of The Classification Theorem}
\label{secapp}

In Section~\ref{sech0} we discuss a spectral sequence of Davis--Mahowald that computes \\ $\Ext_{\CA(1)}(M, \FF_2)$ for an $\CA(1)$-module, $M$ \cite{DM}.  Utilizing this spectral sequence to compute \\ $h_0^{-1}\Ext_{\CA(1)}(M, \FF_2)$ generalizes a previously known formula for computing such localizations when the $\CA(1)$-modules involved have a compatible $\CA$-module structure \cite{Davis}.  In Section~\ref{seclifting} we discuss the consequences of this spectral sequence for determining whether an $\CA(1)$-module can be lifted to an $\CA$-module.  In both of these sections, the classification theorem significantly simplifies some computations, though this often depends on being able to give an explicit decomposition of a  bounded below $Q_0$-local $\CA(1)$-module of finite type into a direct sum of seagull modules. 

\subsection{\texorpdfstring{{Computing $h_0^{-1}\Ext_{\CA(1)}$}}{Computing localized Ext}}
\label{sech0}
When computing Adams spectral sequences, it is often overly ambitious to try to compute the entire $E_2$-page at once. One approach to producing partial descriptions is to isolate periodic families within the $E_2$-page, {or within homotopy groups, as in chromatic localization (see for example \cite{HSNil} and \cite{HPS})}.  In this section, we focus on one type of periodicity in $\Ext_{\CA(1)}^{\bullet, \bullet}(M, \FF_2)$ for $M$, a bounded below $\CA(1)$-module of finite type. 

Recall that $\Ext_{\CA(1)}^{\bullet, \bullet}(\FF_2, \FF_2)$ is the $E_2$-page for the Adams spectral sequence converging to the 2-completed homotopy groups of $ko$ and $h_0$ is the nonzero class in $\Ext_{\CA(1)}^{1,1}(\FF_2, \FF_2)$. 
For any $\CA(1)$-module, $M$, $\Ext_{\CA(1)}^{\bullet, \bullet}(M, \FF_2)$ is a module over $\Ext_{\CA(1)}^{\bullet, \bullet}(\FF_2, \FF_2)$, so there is an $h_0$-action on $\Ext_{\CA(1)}^{\bullet, \bullet}(M, \FF_2)$.  

\FloatBarrier
The goal of this section is to
describe a technique for computing  $h_0^{-1}\Ext_{\CA(1)}^{\bullet, \bullet}(\emd, \FF_2)$ for arbitrary $\CA(1)$-modules via the Davis--Mahowald spectral sequence.  While we do not give a full description of all differentials in the spectral sequence, there are still many cases where they can be computed.

In \cite{DM}, Davis and Mahowald define a graded algebra $R_n^\bullet$ that is used to construct the following spectral sequence. 

\begin{prop}[Davis--Mahowald, \cite{DM}] 
For any $\mathcal{A}(n)$-module $M$, there is a spectral sequence converging to $\Ext_{\CA(n)}(M, \FF_2)$ with \[E_1^{\sigma, s, t}=\Ext_{\CA(n-1)}\left((R_n^\sigma )^*\otimes M, \FF_2\right).\]
\end{prop}

Here, we describe the construction of the Davis--Mahowald spectral sequence in the case where $n=1$ and construct $N_\bullet := (R_1^\bullet)_*$ directly. The use of this spectral sequence for the computation of $h_0^{-1}\Ext_{\CA(1)}^{\bullet, \bullet}(\text{---}, \FF_2)$ was suggested to the author by John Rognes, and the following development of the spectral sequence relies heavily on upcoming work of Rognes and Bruner \cite{BR}.  

Let 
\[ N = \FF_2[x_2, x_3] \]
where the degree of $x_i$ is $(s,t)=(1,i)$ and let $N_\sigma$ be the homogeneous polynomials in $N$ of polynomial degree $\sigma$.  (Note that this does not align with the notation in \cite{DM}, as we choose instead to have the indices coincide with the degree of each generator.)

Give $N$ the structure of an $\CA(1)$-module as follows:

   \begin{align*}
         Sq^1\left(x_2^i x_3^j\right) &= \begin{cases} x_2^{i-1}x_3^{j+1} & i > 0 ,  \ j \text{ even} \\
         0 & \text{ otherwise,} \end{cases} \\
         Sq^2\left(x_2^i x_3^j\right) &= \begin{cases} x_2^{i-2}x_3^{j+2} & i>1, \  j \equiv 0,1 \pmod{4} \\
         0 & \text{otherwise.} \end{cases} \\
    \end{align*}
Note that $N_\sigma$ is then a submodule of $N$. A portion of the module $N$ is shown in Figure~\ref{figN}.

\begin{figure}[ht!]
    \centering
    \begin{tikzpicture}[scale=0.3]
     \foreach \i in {0,1,2,3, 4, 5}{
        \foreach \j in {0,...,\i}{
        \fill (6*\i, 2*\i+\j) circle (4pt);
        }
        \node [draw=none, fill=none, below] at (6*\i, -0.1){$N_{\i}$};
     }
     \node [draw=none, fill=none, above] at (0,0){{1}};
     \foreach \i in {2,3,4,5}{
     \node [draw=none, fill=none, above] at (6*\i, 3*\i){{$x_3^{\i}$}};
     \node [draw=none, fill=none, below] at (6*\i, 2*\i){{$x_2^{\i}$}};
     }
     \node [draw=none, fill=none, above] at (6, 3){{$x_3$}};
     \node [draw=none, fill=none, below] at (6, 2){{$x_2$}};
     \node [draw=none, fill=none, right] at (12, 5){{$x_2 x_3$}};
     \foreach \i in {1, 2, 3, 4, 5}{
     \sqone(6*\i, 2*\i,);
     }
     \foreach \i in {3, 4, 5}{
     \sqone(6*\i, 2*\i+2,);
     \sqtwoR(6*\i, 2*\i+1,);
     }
     \sqone(30, 14,);
     \foreach \i in {2, 3, 4, 5}{
     \sqtwoL(6*\i, 2*\i,);
     }
      \node [draw=none, fill=none, right] at (18, 7){{$x_2^2 x_3$}};
      \node [draw=none, fill=none, left] at (17.6, 8){{$x_2 x_3^2$}};
\end{tikzpicture}
    \caption{The module, $N$, for $\sigma \leq 5$.  The grading by $\sigma$ is depicted horizontally, and the grading by $t$ is depicted vertically. }
    \label{figN}
\end{figure}
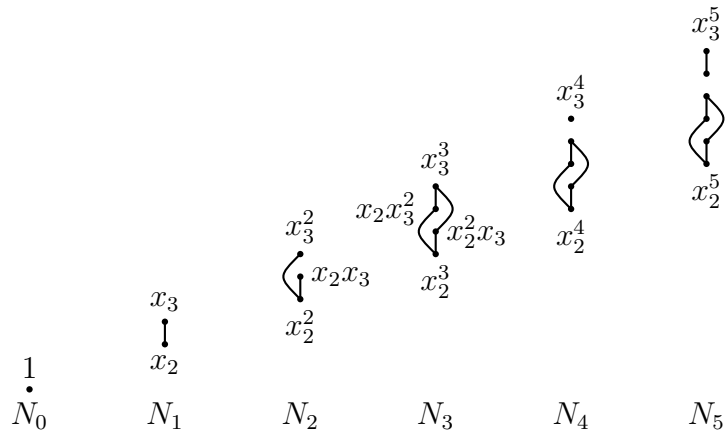
\FloatBarrier
The sequence,
\[  0 \leftarrow \FF_2 \xleftarrow{\partial_0}  \Upsilon_1 \xleftarrow{\partial_1} \Upsilon_1 \otimes N_1 \xleftarrow{\partial_2} \Upsilon_1 \otimes N_2 \xleftarrow{\partial_3} \Upsilon_1 \otimes N_3 \leftarrow \cdots  \]
is exact. For any $\CA(1)$-module, $M$ we have a short exact sequence
\[  0 \to \im(\partial_{\sigma+1})\otimes M \to \Upsilon_1 \otimes N_\sigma \otimes M \to \im(\partial_{\sigma})\otimes M \to 0 \]
for each $\sigma \geq 0$, and a resulting exact couple,
\begin{equation}\label{eqcouple} \begin{tikzcd}[]
\Ext_{\CA(1)}^{s-\sigma-1, t}(\im(\partial_{\sigma+1}) \otimes M, \FF_2) \ar[r, "\delta"] & \Ext_{\CA(1)}^{s-\sigma, t}(\im(\partial_\sigma) \otimes M, \FF_2) \ar[d, "j"] \\
&\Ext_{\CA(1)}^{s-\sigma, t}(\Upsilon_1 \otimes N_\sigma \otimes M, \FF_2) \ar[lu, dashed, "i"]
\end{tikzcd}\end{equation}

Applying a change of rings isomorphism, we see \[ \Ext_{\CA(1)}^{s-\sigma, t}(\Upsilon_1 \otimes N_\sigma \otimes M, \FF_2) = \Ext_{\CA(1)}^{s-\sigma, t}(\CA(1) \otimes_{\CA(0)} \FF_2 \otimes N_\sigma \otimes M, \FF_2) \cong \Ext_{\CA(0)}^{s-\sigma, t}(N_\sigma \otimes M, \FF_2), \]
where $\CA(1) \otimes_{\CA(0)} \FF_2$ is an $\CA(1)$-module via the left action on $\CA(1)$.
The  associated to the exact couple, \eqref{eqcouple}, can therefore be written as
\[E_1^{\sigma, s, t}=\Ext_{\CA(0)}^{s-\sigma,t}(N_\sigma \otimes M, \FF_2) \Rightarrow \Ext_{\CA(1)}^{s-\sigma, t}(M, \FF_2).\]
The differentials have the form $d_r: E_r^{\sigma, s, t} \to E_r^{\sigma+r, s+1, t}$. When depicting $\Ext^{s, t}$ with the Adams grading, we will use the fact that the degree of $d_r$ is $(\sigma, s, t-s) = (r, 1, -1) $.

In order to compute $h_0^{-1}\Ext_{\CA(1)}(M, \FF_2)$, we consider the Davis--Mahowald spectral sequence with $h_0$ inverted.  

\begin{lemma}\label{lemlocalDM}
Let $M$ be any $\CA(1)$-module.  The $h_0$-localization of the Davis--Mahowald spectral sequence for $M$ has the form,
\[ E_1^{\bullet, \bullet, \bullet} = P \otimes H_\bullet (M_*; Q_0) \Rightarrow h_0^{-1}\Ext_{\CA(1)}^{\bullet, \bullet}(M, \FF_2), \]
where $P$ is isomorphic, as a trigraded $\FF_2$-vector space, to $\FF_2[h_0^{\pm 1}, x_3^2]$.
\end{lemma}

Here, $M_*$ is the dual of $M$ as an $\FF_2$-vector space with a right $\CA(1)$-module structure given by precomposition.  The $Q_0$-homology of $M_*$ is defined with respect to the right action of $Q_0$.  We take the grading on $M_*$ to be \[ (M_*)_k = \Hom^k\left(M, \FF_2 \right) = \Hom_{\FF_2}( M, \Sigma^k \FF_2). \]  
\begin{note}
As a module over $E[Q_0]=\CA(0)$, any $\CA(1)-$module $M$ is isomorphic to a direct sum of suspensions of $\FF_2$ and suspensions of $\CA(0)$.  Using this decomposition, it is quick to show that $H_\bullet (M;Q_0) \cong H_\bullet (M_*;Q_0)$.  However, using $H_\bullet(M_*;Q_0)$ allows for a clearer statement of the differentials.  
\end{note}

The content of the proof of Lemma~\ref{lemlocalDM} is identifying the $E_1$-page, which we accomplish in Lemmas~\ref{lemA0} and~\ref{thmE1}.  

\begin{lemma} \label{lemA0}
For any $\CA(1)$-module, $M$,
\[ h_0^{-1}\Ext_{\CA(0)}^{\bullet, \bullet}(M, \FF_2) \cong \FF_2[h_0^{\pm 1}] \otimes  H_{\bullet}(M_*; Q_0). \]
\end{lemma}
\begin{proof}
As an $\CA(0)$-module 
\[ M = \left(\bigoplus_{b \in B} \FF_2 \{b\}\right) \oplus \left(\bigoplus_{c \in C} \CA(0)\{c\} \right), \]
for some $B, C \subset M$.

We can compute $\Ext_{\CA(0)}^{\bullet, \bullet}(M, \FF_2)$ using this decomposition of $M$:
\begin{align*}
    \Ext_{\CA(0)}^{\bullet,\bullet}\left(M, \FF_2 \right) &\cong \left(\bigoplus_{b \in B} \Ext_{\CA(0)}^{\bullet,\bullet}(\FF_2\{b\},\FF_2)  \right) \oplus \left(\bigoplus_{c \in C} \Sigma^{|c|}\FF_2 \right).
\end{align*}
We consider the injective resolution $\FF_2 \to J_\bullet$ where $J_{s}=\Sigma^{-s-1}\FF_2$ and the maps $\ell_s:J_{s-1} \to J_s$ are the unique nonzero maps. 
Any class in $\Ext_{\CA(0)}^{s-\sigma, t}(\FF_2\{b\}, \FF_2)$ is represented by a map \[g:\Sigma^{-t}\FF_2\{b\} \to J_{s-\sigma}\] such that $\ell_{s-\sigma +1} \circ g =0$.  The image of $g$ is thus contained in the top degree of $J_{s-\sigma}$. Such a map factors as a composition,
\[ \Sigma^{-t}\FF_2\{b\} \xrightarrow{\wh{g}} \Sigma^{-s+\sigma}\FF_2 \xrightarrow{h_0^{s-\sigma} \ell_0} J_{s-\sigma} \]

The map $h_0^{s-\sigma} \ell_0$ takes the nonzero class in $\FF_2$ to the top class in $J_{s-\sigma}$.  (Note that this is, in fact, the map that results from acting on $\ell_0$ by $h_0^{s-\sigma}$.)
If $g$ is nonzero, then the degree of the nonzero class in $\Sigma^{-t}\FF_2\{b\}$, $|b|-t$, must be the same as the degree of the top class in $J_{s-\sigma}$, $-s+\sigma$.  So, $t-s+\sigma=|b|$.  The map $\wh{g}$ is the unique nonzero one.  In
\[  \bigoplus_{b \in B}  \Hom_{\FF_2}^{\bullet}(\FF_2\{b\},\FF_2) \subseteq \Hom_{\FF_2}^\bullet(M, \FF_2), 
\]
the map $\wh{g}$ is $b_*$.   
So, 
\[ \bigoplus_{b \in B} \Ext_{\CA(0)}^{\bullet, \bullet}(\FF_2\{b\}, \FF_2) \cong \FF_2[h_0] \otimes  \bigoplus_{b \in B} \FF_2[h_0]\{b_*\} \cong H_{\bullet}(M_*; Q_0) , 
\]
and thus 
\begin{align*}
    \Ext_{\CA(0)}^{\bullet,\bullet}\left(M, \FF_2 \right) &\cong  \left(\FF_2[h_0] \otimes H_\bullet(M_*;Q_0)   \right) \oplus \left(\bigoplus_{c \in C} \Sigma^{|c|}\FF_2 \right).
\end{align*}
When $h_0$ is inverted, this gives us 
\begin{align*}
    h_0^{-1}\Ext_{\CA(0)}^{\bullet,\bullet}\left(M, \FF_2 \right) 
    &\cong \FF_2[h_0^{\pm 1}] \otimes H_\bullet(M_*;Q_0) 
\end{align*}
with 
\[ h_0^{-1}\Ext_{\CA(0)}^{s-\sigma, t}(M, \FF_2) \cong H_{t-(s-\sigma)}(M)\{h_0^{s-\sigma}\} \]
and we note that inverting $h_0$ is an isomorphism in all degrees $s>0$.  
\end{proof}

\begin{thm}\label{thmE1}
The $E_1$-page of the $h_0$-localized Davis--Mahowald spectral sequence for an $\CA(1)$-module, $M$, is isomorphic, as a trigraded $\FF_2$-vector space, to 
\[ \FF_2[h_0^{\pm 1}, x_3^2] \otimes H_\bullet(M_*;Q_0).\]
\end{thm}

\begin{note}
We emphasize that this is only an isomorphism of vector spaces. In the case where $M$ is an $\CA(1)$-module coalgebra, giving $\Ext_{\CA(0)}^{\bullet, \bullet}(N_\bullet \otimes M, \FF_2)$ the structure of an algebra, the Davis--Mahowald spectral sequence is a spectral sequence of algebras and the isomorphism in Theorem~\ref{thmE1} is an isomorphism of algebras.  
\end{note}

\begin{proof}
The $Q_0$-homology of $N_\bullet$ is 
\[ H_k (N_\sigma;Q_0) = \begin{cases} 
\FF_2\{x_3^\sigma\} & \sigma \text{ even, } k=3\sigma\\
0 & \sigma \text{ otherwise.}
\end{cases} \]
This identifies the $E_1$-page,
\begin{align*} E_1^{\sigma, s, t} &\cong h_0^{-1}\Ext_{\CA(0)}^{s-\sigma, t}\left( N_\sigma \otimes M, \FF_2 \right) \\ &\cong   H_{t-(s-\sigma)}\left((N_\sigma \otimes M)_*; Q_0\right)\{h_0^{\pm 1}\}   \\  &\cong  H_{3\sigma}\big(N_\sigma; Q_0\big) \otimes H_{t-s-2\sigma}(M_*; Q_0)\{h_0^{\pm 1}\} \\
& \cong \begin{cases} H_{t-s-2\sigma}((M)_*; Q_0)\{h_0^{\pm 1}, x_3^\sigma\} & \sigma \text{ is even} \\
0 & \sigma \text{ is odd }\end{cases} \\
&\cong \left(\FF_2[h_0^{\pm 1}, x_3^2] \otimes H_\bullet(M_*; Q_0) \right)_{\sigma, s, t} . \qedhere \end{align*}
\end{proof}
Notice $E_1^{\sigma, s, t}$ is zero when $\sigma$ is odd. The degree of $d_r$ is $(\sigma, s, t-s)=(r, 1, -1)$, so nonzero differentials will only occur when both $\sigma$ and $r$ are even.
In computing the first potentially nontrivial differential, $d_2$, we will want to utilize an explicit description of the exact sequence, 
\[  0 \leftarrow \FF_2 \xleftarrow{\partial_0}  \Upsilon_1 \xleftarrow{\partial_1} \Upsilon_1 \otimes N_1 \xleftarrow{\partial_2} \Upsilon_1 \otimes N_2 \xleftarrow{\partial_3} \Upsilon_1 \otimes N_3 \leftarrow \cdots  \]
There is a decomposition,
\[ \Upsilon_1 \otimes N_\sigma = A_\sigma \oplus B_\sigma \]
where 
\begin{align*}  A_\sigma &= \begin{cases} \CA(1)\{1 \otimes x_2^{\sigma-2j}x_3^{2j} | 0 \leq j \leq 2 \lfloor \sigma/4 \rfloor - 1 \} \oplus \CA(1)\{1 \otimes x_2 x_3^{\sigma-1}\} & \sigma \equiv 3 \pmod{4} \\
\CA(1)\{1 \otimes x_2^{\sigma-2j}x_3^{2j} | 0 \leq j \leq 2 \lfloor \sigma/4 \rfloor - 1 \} & \text{otherwise,}
\end{cases} \\
B_\sigma &= \begin{cases}
\Upsilon_1\{1 \otimes x_3^\sigma\} & \sigma \equiv 0 \pmod{4} \\
\CA(1)\{1 \otimes x_2 x_3^{\sigma-1}\} & \sigma \equiv 1 \pmod{4} \\ 
\CA(1)\{1 \otimes x_2^2 x_3^{\sigma-2}\} \oplus \Upsilon_1\{1 \otimes x_3^\sigma\} & \sigma \equiv 2 \pmod{4} \\
\CA(1)\{1 \otimes x_2^3 x_3^{\sigma-3}\} & \sigma \equiv 3 \pmod{4}.
\end{cases}
\end{align*}
For all $1 \otimes x_2 ^i x_3^{\sigma-i}$ in $A_\sigma$,
\begin{align*}
    \partial_\sigma (1 \otimes x_2^i x_3^{\sigma-i}) = \begin{cases}
    1 \otimes x_2^{i-3}x_3^{\sigma-i+2} & i \equiv \sigma \pmod{4} \\
    0 & \text{otherwise.}
    \end{cases}
\end{align*}
For all $1 \otimes x_2 ^i x_3^{\sigma-i}$ in $B_\sigma$,
\begin{align*}
    \partial_\sigma (1 \otimes x_2^i x_3^{\sigma-i}) = \begin{cases}
    Sq^1 Sq^2 Sq^1 Sq^2 (1 \otimes x_2^3 x_3^{\sigma-4}) & \sigma \equiv 0 \pmod{4}  \\
    Sq^2 (1 \otimes x_3^{\sigma-1}) & i>0, \ \sigma \equiv 1, 2 \pmod{4} \\
    Sq^2 Sq^2 (1 \otimes x_2 x_3^{\sigma-2}) & i=0, \ \sigma \equiv 2 \pmod{4}  \\
    Sq^2(1 \otimes x_2^2 x_3^{\sigma-3}) + 1 \otimes x_3^{\sigma-1} &  \sigma \equiv 3 \pmod{4} \\ 
    \end{cases}
\end{align*}
Note that $\partial_\sigma$ restricts to maps $A_\sigma \to A_{\sigma-1}$ and $B_\sigma \to B_{\sigma-1}$.  When computing $d_2$, we will only have to consider $\partial_\sigma$ on $B_\sigma$, so we give a diagram of the exact sequence \[B_{4k} \xleftarrow{\partial_{4k+1}} B_{4k+1} \xleftarrow{\partial_{4k+2}} B_{4k+2} \xleftarrow{\partial_{4k+3}} B_{4k+3} \xleftarrow{\partial_{4k+4}} B_{4k+4}  \] in Figure~\ref{figseagulltensorN}.
\begin{figure}[!htbp]
    \centering
    \begin{tikzpicture}[scale=0.43]
   \draw[->] (3.8,0)--(0.2,0);
    \SgR(4,0);
    \fill[gray] (6,2) circle (3pt);
    \fill[gray] (6,3) circle (3pt);
    \fill[gray] (6,5) circle (3pt);
    \sqone(6,2,gray);
    \sqtwoR(6,3,gray);
    \Aone(10,2);
    \fill[gray] (13,4) circle (3pt);
    \fill[gray] (13,5) circle (3pt);
    \fill[gray] (15,6) circle (3pt);
    \fill[gray] (15,7) circle (3pt);
    \fill[gray] (15,8) circle (3pt);
    \sqone(13,4,gray);
    \sqone(15,7,gray);
    \sqtwoCR(13,4,gray);
    \sqtwoCR(13,5,gray);
    \sqtwoR(15,6, gray);
    \Aone(19,4);
    \SgR(23,6);
    \draw[->] (9.8,2)--(6.2,2);
    \draw[->] (18.8,4)--(13.2,4);
    \draw[->] (22.8,6).. controls (18.5, 5.5).. (15.2,6);
    \fill[gray] (27,6) circle (3pt);
    \fill[gray] (27,7) circle (3pt);
    \fill[gray] (27,8) circle (3pt);
    \fill[gray] (27,9) circle (3pt);
    \fill[gray] (29,9) circle (3pt);
    \fill[gray] (29,10) circle (3pt);
    \fill[gray] (29,11) circle (3pt);
    \sqone(27,6, gray);
    \sqone(27,8, gray);
    \sqone(29,9, gray);
    \sqtwoL(27,6,gray);
    \sqtwoCR(27,7,gray);
    \sqtwoCR(27,8,gray);
    \sqtwoCR(27,9,gray);
    \Aone(32, 6);
    \draw[->] (31.8,6)--(27.2,6);
    \fill[gray] (0,0) circle (3pt);
    \node at (4,-0.8) {\footnotesize $1 \otimes x_3^{4k}$};
    \node at (10.5,1.2) {\footnotesize $1 \otimes x_2 x_3^{4k}$};
    \node at (19.5,3.2) {\footnotesize $1 \otimes x_2^{2}x_3^{4k}$};
    \node at (22,5.2) {\footnotesize $1 \otimes x_3^{4k+2}$};
    \node[gray] at (27.5,4.7) {\footnotesize \begin{tabular}{c}$Sq^2(1 \otimes x_2^{2}x_3^{4k})$ \\$ + 1 \otimes x_3^{4k+2}$\end{tabular}};
    \node at (32.9,5.2) {\footnotesize $1 \otimes x_2^{4k+3}$};
    \node[gray] at (0,-2) {\footnotesize $\im (\partial_{4k})$};
    \node at (2.8,-2) {\footnotesize $B_{4k}$};
    \node[gray] at (6,-2) {\footnotesize $\supseteq \im (\partial_{4k+1})$};
    \node at (10.8,-2) {\footnotesize $B_{4k+1}$};
    \node[gray] at (14.6,-2) {\footnotesize $\supseteq \im (\partial_{4k+2})$};
    \node at (22,-2) {\footnotesize $B_{4k+2}$};
    \node[gray] at (27,-2) {\footnotesize $\supseteq$ \hspace{0.5em}  $ \im (\partial_{4k+3})$};
    \node at (32.3,-2) {\footnotesize $B_{4k+3}$};
    \draw[->] (2, -2)--(1.5,-2);
    \draw[->] (9.6,-2)--(8.5,-2);
    \draw[->] (20.8,-2)--(17.2,-2);
    \draw[->] (31,-2)--(30,-2);
    \end{tikzpicture}
    \caption{The resolution $B_\bullet$.  The maps $\partial_\sigma$ are drawn only for the generators.  Since each $\partial_\sigma$ is an $\CA(1)$-map, this determines the value of $\partial_\sigma$ on any class.}
    \label{figseagulltensorN}
\end{figure}
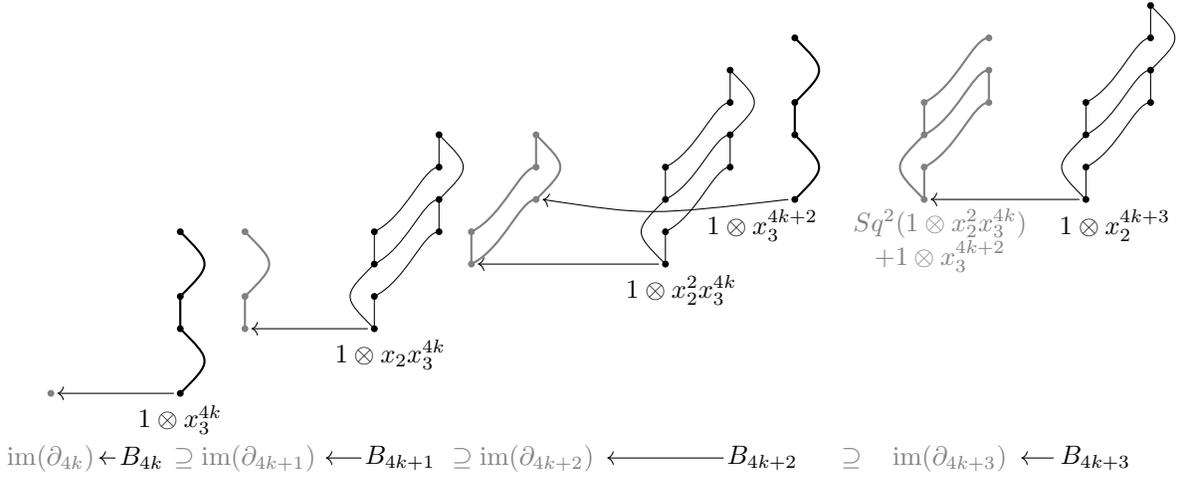

\FloatBarrier

We will also need an injective resolution of $\FF_2$ over $\CA(1)$.  We choose the following:
\begin{align*}
    I_s = \begin{cases}
    \CA(1)\{e_{s, -s-4j-6} \ |\ 0 \leq j \leq s/2 \} & s \equiv 0 \pmod{4} \\
    \CA(1)\{e_{s, -s-4j-6} \ |\ 0 \leq j \leq (s-1)/2 \} \oplus \CA(1)\{ e_{s, -5-3s} \} & s \equiv 1 \pmod{4} \\
    \CA(1)\{e_{s, -s-4j-6} \ |\ 0 \leq j \leq (s-2)/2 \} \oplus \CA(1)\{ e_{s, -4-3s} \} & s \equiv 2 \pmod{4} \\
    \CA(1)\{e_{s, -s-4j-6} \ |\ 0 \leq j \leq (s-1)/2 \} & s \equiv 3 \pmod{4} \\
    \end{cases}
\end{align*}
where $|e_{s,k}|=k$.  
The maps $f_s: I_{s-1} \to I_s$ are the $\CA(1)$-maps determined by,
\begin{align*}
    f_s(e_{s-1,k}) &= \begin{cases}
    Sq^1 e_{s, k-1}+Sq^2 e_{s, k-2} & k=-3s-3, s \equiv 1 \pmod{4} \\
        Sq^1 e_{s,k-1} +Sq^2 Sq^1 e_{s, k-3} & k=-3s-1, s \equiv 2 \pmod{4} \\
    Sq^2 e_{s,k-2} & k=-3s-2, s \equiv 2 \pmod{4} \\
    Sq^2 Sq^1 e_{s,k-3} & k=-3s-4, s \equiv 3 \pmod{4} \\
     Sq^1 e_{s, k-1} + Sq^2 Sq^1 Sq^2 e_{s, k-5} & \text{otherwise.}
    \end{cases}
\end{align*}

More efficiently, we could say \[f_s(e_{s-1, k})= Sq^1 e_{s, k-1}+Sq^2 e_{s, k-2} + Sq^2 Sq^1 e_{s, k-3}+Sq^2 Sq^1 Sq^2 e_{s,k-5},\] omitting any terms where the required generators do not exist in $I_s$.

\begin{figure}[ht!]
    \centering
    \caption{The beginning of the $\CA(1)$-resolution $\FF_2 \to I_\bullet$. An arrow that hits an arc between circled classes indicates that $f_s$ on the class where the arrow originates is the sum of the circled classes. As in Figure~\ref{figseagulltensorN}, the values of the maps are only indicated for generators.}
    \label{figF2res}
    \begin{tikzpicture}[scale=0.4]
        \fill(0,0) circle (3pt);
        \Aonecolor(2,-6, gray);
        \node[below] at (2,-6){$e_{0,-6}$};
        \node[above] at (4,0){$t_0$};
        \draw[->](0.2,0)--(3.8,0);
        \Aonecolor(6,-7, gray);
        \Aonecolor(8,-8, gray);
        \node[below] at (6,-7){$e_{1,-7}$};
        \node[below] at (8,-8){$e_{1,-8}$};
        \node[above] at (8,-1){$t_1$};
        \draw(6,-6) circle (8pt);
        \draw(8,-6) circle (8pt);
        \draw(6.2,-6) .. controls(7,-6.25)..(7.8,-6);
        \draw[->](2.2,-6).. controls(3,-6.5) and (6,-6.5)..(6.6,-6.25);
        \Aonecolor(12, -8, gray);
        \Aonecolor(13,-10, gray);
        \node[left] at (12,-8){$e_{2,-8}$};
        \node[below] at (13,-10){$e_{2,-10}$};
        \node[above] at (14,-2){$t_2$};
        \draw(12,-7) circle (8pt);
        \draw(15,-7) circle (8pt);
        \draw(12.2,-7) .. controls(13,-7.25)..(14.8,-7);
        \draw[->](6.2,-7).. controls(7,-7.5) and (11,-7.5)..(12.7,-7.25);
        \draw[->](8.2,-8)..controls(11,-8.5)..(12.8,-8);
        \Aonecolor(18,-9, gray);
        \Aonecolor(18,-13, gray);
        \node[left] at (18,-9){$e_{3,-9}$};
        \node[below] at (18,-13){$e_{3,-13}$};
        \node[above] at (20,-3){$t_3$};
        \draw(18,-8) circle (8pt);
        \draw(20,-8) circle (8pt);
        \draw(18.2,-8) .. controls (19, -7.75)..(20, -8);
        \draw[->](12,-8)..controls (13, -7.25) and (18, -7.25)..(18.8, -7.75);
        \draw[->](13.2,-10)..controls (15,-9.5) and (18,-9.5)..(19.8,-10);
        \Aonecolor(24,-10, gray);
        \Aonecolor(24,-14, gray);
        \Aonecolor(24, -18, gray);
        \node[left] at (24,-10){$e_{4,-10}$};
        \node[left] at (24,-14){$e_{4,-14}$};
        \node[below] at (24, -18){$e_{4,-18}$};
        \node[above] at (26, -4){$t_4$};
        \draw(24,-9) circle (8pt);
        \draw(26,-9) circle (8pt);
        \draw(24.2,-9) .. controls (25, -8.75)..(26, -9);
        \draw[->](18,-9)..controls (19, -8.25) and (24, -8.25)..(24.8, -8.75);
                \draw(24,-13) circle (8pt);
        \draw(26,-13) circle (8pt);
        \draw(24.2,-13) .. controls (25, -12.75)..(26, -13);
        \draw[->](18,-13)..controls (19, -12.25) and (24, -12.25)..(24.8, -12.75);
    \end{tikzpicture}
\end{figure}
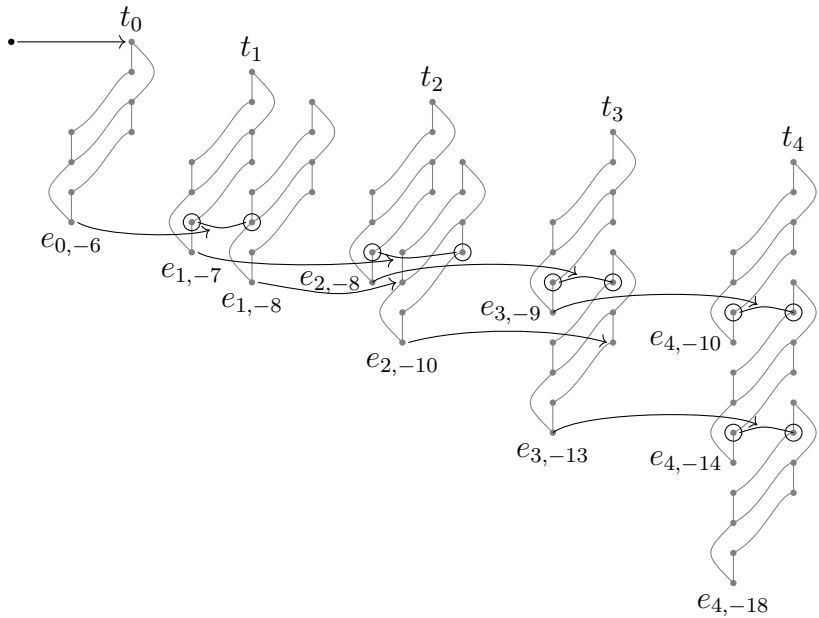

Let $r_s= Sq^2 Sq^1 Sq^2 e_{s, -s-6}$ and $t_s = Sq^1 Sq^2 Sq^1 Sq^2 e_{s, -s-6}$.  Then \\ $Sq^1 r_s = t_s$ and $f_{s}(t_{s-1})=r_s$ for $s \geq 1$.  (When $s=0$, we have $f_0(1)=t_0$.)  Note that $h_0^s f_0$ is the unique map $\FF_2 \to I_s$ that takes $1$ to $t_s$.
\FloatBarrier

Consider the following distinguished classes on the $E_1$-page.  \begin{defn}\label{defnphib}
Let $\varphi_m^{\sigma}: \Upsilon_1 \otimes N_\sigma \otimes M \to \Sigma^{|m|}\FF_2$ be the unique $\CA(1)$-map that takes $1 \otimes a \otimes b$ to $(x_3^\sigma)_*(a)m_*(b)$.
\end{defn}
The class $[\varphi_m^{\sigma} \cdot t_{s-\sigma}]$ in $\Ext_{\CA(1)}^{s-\sigma,|m|}\left(\Upsilon_1 \otimes M, \FF_2\right)$ is the image of \[h_0^{s-\sigma} x_3^\sigma [m_*] \in \FF_2[h_0^{\pm 1}, x_3^{2}] \otimes H_\bullet (M_*;Q_0) \] under the shearing and change of rings isomorphisms. Let $\CB$ be any set of elements in $M$ such that $\{[m] | m \in \CB\}$ is an $\FF_2$-basis for $H_\bullet (M; Q_0)$.  Then $\{[\varphi_{m}^{\sigma} \cdot t_s] |  m \in \CB, \sigma \geq 0  \}$ is an $\FF_2$-basis for $h_0^{\pm 1} \Ext_{\CA(0)}^{\bullet, \bullet}(N_\bullet \otimes M, \FF_2)$.

In computing $d_2$, we will apply the following lemma to $\varphi_m^{\sigma}$ and related maps. 
\begin{lemma}\label{lemF2map}
Let $M, N$ be $\CA(1)$-modules.

For any $\CA(1)$-map, $f: N \otimes M \to \FF_2$,
\[ f\left( Sq^2 Sq^1 Sq^2 n \otimes m\right) = f\left(n \otimes Sq^2 Sq^1 Sq^2 m \right) \]
for all $n \in N$ and all $m \in M$.
\end{lemma}
\begin{proof}
For all $n \in N$ and $m \in M$,
\begin{align*}
    Sq^2 Sq^1 Sq^2 \left( n \otimes m\right) =&  Sq^2 Sq^1 Sq^2 n \otimes m + Sq^1\left(Sq^2 n \otimes Sq^2 m \right) \\ &+ Sq^2\left( Sq^2n \otimes Sq^1 m + Sq^1 n \otimes Sq^2 m \right) + n \otimes Sq^2 Sq^1 Sq^2 m.
\end{align*}
Since $Sq^1$ and $Sq^2$ act trivially on $\FF_2$ and $f$ is an $\CA(1)$-map, \[f(Sq^1 x)=f(Sq^2 x) = 0\] for all $x \in N \otimes M$.  Thus,
\[ 0 = f\left(Sq^2 Sq^1 Sq^2 n \otimes m\right) + f\left(n \otimes Sq^2 Sq^1 Sq^2 m\right). \qedhere \]
\end{proof}

We can now compute the second differential. 
\begin{thm}\label{thmd1}
The second differential is given by
\[ d_2\left(h_0^{s-\sigma} x_3^{\sigma} [b_*] \right) = h_0^{s-\sigma-1}x_3^{\sigma+2}  \left[ Sq^2 Sq^1 Sq^2 b_* \right] \]
for any nonzero $h_0^{s-\sigma} x_3^{\sigma} [b_*]  \in  \FF_2[h_0^{\pm 1}, x_3^2] \otimes H_\bullet (M_*;Q_0) $.
\end{thm}  

\begin{note}
In the case where $M$ is a coalgebra, the proof of this formula is simplified by the use of the Leibniz rule.  
\end{note}

\begin{proof}

Consider some nonzero $h_0^{s-\sigma} x_3^{\sigma}[b_*]  \in H_\bullet (M_*; Q_0) \otimes \FF_2[h_0^{\pm 1}, x_3^2]$.  Then $\sigma$ is even and $Sq^1 b_* = 0 = Sq^1 b$.  If $s\geq \sigma$, the class $h_0^{s-\sigma} x_3^{\sigma} [b_*]$  is represented by \[\varphi_b^{\sigma} \cdot t_{s-\sigma}: \Upsilon_1 \otimes N_\sigma \otimes M \to I_{s-\sigma}.\]  We unravel the exact couple:
\[\adjustbox{scale=0.93}{\begin{tikzcd}[column sep=tiny]
  \Ext_{\CA(1)}^{s-\sigma-1,\bullet}\left(\im(\partial_{\sigma+2}) \otimes M, \FF_2 \right) \ar[r, "\delta"] \ar[d, "j"]  & \Ext_{\CA(1)}^{s-\sigma,\bullet}\left(\im(\partial_{\sigma+1}) \otimes M, \FF_2\right) &  \\
\Ext_{\CA(1)}^{s-\sigma-1,\bullet}\left(\Upsilon_1 \otimes N_{\sigma+2} \otimes M, \FF_2 \right)    & &  \Ext_{\CA(1)}^{s-\sigma,\bullet}\left(\Upsilon_1 \otimes N_{\sigma} \otimes M, \FF_2 \right) \ar[ul, "i"] \ar[ll, dashed, "d_2"]
\end{tikzcd}}\]  

The map $i$ in the exact couple is just precomposition by the inclusion of $\im(\partial_{\sigma+1}) \otimes M$ into $ \Upsilon_1 \otimes N_\sigma \otimes M$.  We will refer to the composition of $\varphi_b^\sigma \cdot t_{s-\sigma}$ with this inclusion as $\varphi_b^\sigma \cdot t_{s-\sigma}$ as well.  

Then, for any map $\Phi:\im(\partial_{\sigma+2}) \otimes M \to I_{s-\sigma}$ such that $\delta[\Phi] = [\varphi_b^{\sigma} \cdot t_{s-\sigma}] $, $d_2[\varphi_b^\sigma \cdot t_{s-\sigma}] = j[ \Phi]$ where $j$ is precomposition with $\partial_{\sigma+2} \otimes \text{id}_M$.  We will not give a full description of $\Phi$.  Instead, we will use the fact that $\{[\varphi_m^{\sigma+2} \cdot t_{s-\sigma}] | m \in \CB\} $ is a basis for $h_0^{\pm 1} \Ext_{\CA(1)}^{s-\sigma, \bullet}(\Upsilon_1 \otimes N_{\sigma+2} \otimes M, \FF_2)$.  Since $\varphi_m^{\sigma+2}$ is the unique $\CA(1)$-map with $\varphi_m^{\sigma+2}(1 \otimes a \otimes n)=(x_3^{\sigma+2})_*(a)m_*(n)$, it is sufficient to determine the value of $j[\Phi]$ on terms of the form $1 \otimes x_3^{\sigma+2} \otimes m$ for $m \in \CB$.   

Note that $j[\Phi](1 \otimes x_3^{\sigma+2} \otimes m)=\Phi(\partial_{\sigma+2}(1 \otimes x_3^{\sigma+2}) \otimes m)$.  The expression $\partial_{\sigma+2}(1 \otimes x_3^{\sigma+2})$ depends on the equivalence class of $\sigma$ modulo four:
\[ \partial_{\sigma+2} (1 \otimes x_3^\sigma) = \begin{cases} Sq^1 \left( Sq^2 Sq^1 (1 \otimes x_2^{\sigma+1})\right) & \sigma \equiv 0 \pmod{4} \\ 
Sq^1 \left( Sq^2 Sq^1 Sq^2 (1 \otimes x_2^{\sigma+1})\right)  & \sigma \equiv 2 \pmod{4}
\end{cases} \]
In either case,
\( \partial_{\sigma+2}(1 \otimes x_3^\sigma) = Sq^1 \left( y \right) \) where $\partial_{\sigma+1}(y) = Sq^2 Sq^1 Sq^2 (1 \otimes x_3^\sigma)$. This will be enough to compute $j[\Phi](1 \otimes x_3^{\sigma+2} \otimes m)$.

The connecting homomorphism, $\delta$, is constructed by applying the Snake Lemma to the diagram below:
\[ \adjustbox{scale=0.95}{\begin{tikzcd}[column sep=tiny]
\Hom\left(\im(\partial_{\sigma+1}) \otimes M, I_{s-\sigma-1}\right) \ar[r] \ar[d] & \Hom\left(\Upsilon_1 \otimes N_{\sigma+1} \otimes M, I_{s-\sigma-1}\right) \ar[r] \ar[d] & \Hom\left(\im(\partial_{\sigma+2}) \otimes M, I_{s-\sigma-1}\right) \ar[d] \\
\Hom\left(\im(\partial_{\sigma+1}) \otimes M, I_{s-\sigma}\right) \ar[r] & \Hom\left( \Upsilon_1 \otimes N_{\sigma+1} \otimes M, I_{s-\sigma}\right) \ar[r] & \Hom\left(\im(\partial_{\sigma+2}) \otimes M, I_{s-\sigma}\right)
\end{tikzcd}}\]
where $\Hom(A,B)=\Hom_{\CA(1)}^{\bullet}(A,B)$.

Fix any $\Phi \in \Hom_{\CA(1)}^\bullet(\im(\partial_{\sigma+2}) \otimes M, I_{s-\sigma-1})$ whose image under the connecting homomorphism is $\varphi_b^{\sigma} \cdot t_{s-\sigma} \in \Hom_{\CA(1)}^\bullet(\im(\partial_{\sigma+1}) \otimes M, I_{s-\sigma})$.  Let $\wt{\Phi} \in \Hom_{\CA(1)}^\bullet(\im(\Upsilon_1 \otimes N_{\sigma+1} \otimes M, I_{s-\sigma-1})$ be a lift of $\Phi$.  Then
\begin{align*}
    \Phi(\partial_{\sigma +2}(1 \otimes x_3^{\sigma+2}) \otimes m) &= \wt{\Phi} (\partial_{\sigma +2}(1 \otimes x_3^{\sigma+2}) \otimes m) = \wt{\Phi} (Sq^1 (y \otimes m)) = Sq^1 \wt{\Phi}(y \otimes m) 
\end{align*}
where $\partial_{\sigma+1}(y) = Sq^2 Sq^1 Sq^2 (1 \otimes x_3^{\sigma})$.

The bottom left horizontal map \[\Hom_{\CA(1)}^{\bullet}\left(\im(\partial_{\sigma+1}) \otimes M, I_{s}\right) \to  \Hom_{\CA(1)}^{\bullet}\left( \Upsilon_1 \otimes N_{\sigma+1} \otimes M, I_{s}\right)\] is given by precomposition with $\partial_{\sigma+1} \otimes \text{id}_M$. The center vertical map \[\Hom_{\CA(1)}^{\bullet}\left( \Upsilon_1 \otimes N_{\sigma+1} \otimes M, I_{s-\sigma-1}\right)\to  \Hom_{\CA(1)}^{\bullet}\left( \Upsilon_1 \otimes N_{\sigma+1} \otimes M, I_{s-\sigma}\right)\] is given by composition with $f_{s-\sigma}: I_{s-\sigma-1} \to I_{s-\sigma}$.  
Hence,
\begin{align*}
    f_{s-\sigma}\left(\wt{\Phi} (y \otimes m)\right) &= (\varphi_b^\sigma \cdot t_{s-\sigma})\left( \partial_{\sigma+1}(y) \otimes m \right) = (\varphi_b^\sigma \cdot t_{s-\sigma})\left( Sq^2 Sq^1 Sq^2 (1 \otimes x_3^\sigma) \otimes m \right).
\end{align*}
Lemma ~\ref{lemF2map} then implies 
\[f_{s-\sigma}\left(\wt{\Phi} (y \otimes m)\right)=({\varphi}_b^\sigma \cdot t_{s-\sigma})\left( 1 \otimes x_3^{\sigma} \otimes Sq^2 Sq^1 Sq^2 m \right),\]
since $\varphi_b^\sigma \cdot t_{s-\sigma}$ factors through $\FF_2$.
So, \begin{align*}
    f_{s-\sigma}\left(\wt{\Phi} (y \otimes m)\right) &=  b_*(Sq^2 Sq^1 Sq^2 m) \cdot t_{s-\sigma} \\
    &= \big((b_*) Sq^2 Sq^1 Sq^2 \big)(m) \cdot f_{s-\sigma} (r_{s-\sigma-1})
\end{align*}
and, since $f_{s-\sigma}$ is injective in degree $|t_{s-\sigma}|$,
\[ \wt{\Phi}(y \otimes m) = \big((b_*) Sq^2 Sq^1 Sq^2 \big)(m) \cdot r_{s-\sigma-1}. \]
Then,
\begin{align*}
    \Phi(\partial_{\sigma+2}(1 \otimes x_3^{\sigma+2}) \otimes m) &= Sq^1 \wt{\Phi}(y \otimes m) \\
    &= Sq^1 \left( \big((b_*)Sq^2 Sq^1 Sq^2 \big)(m) \cdot r_{s-\sigma -1} \right) \\
    &= \big((b_*)Sq^2 Sq^1 Sq^2 \big)(m) \cdot t_{s-\sigma-1} \\
    &= \varphi_{\left((b_*)Sq^2 Sq^1 Sq^2   \right)_*}^{\sigma+2}(1 \otimes x_3^{\sigma+2} \otimes m) \cdot t_{s-\sigma-1}.
\end{align*}
Therefore, $d_2(h_0^{s-\sigma} x_3^\sigma [b_*])=h_0^{s-\sigma-1} x_3^{\sigma+2}[(b_*)Sq^2 Sq^1 Sq^2 ]$.
\end{proof}

Heuristically, the second differential pairs $Q_0$-homology classes that are connected by a 1-seagull in $M_*$.  We conjecture that the $2n^{th}$ differential pairs classes connected by an $n$-seagull. 

\begin{conj}\label{conjhigherd}
There exists a differential
\[ d_{2n} ([x_*]h_0^{s-\sigma} x_3^{\sigma})=[y_*]h_0^{s+1-\sigma - 2n} x_3^{\sigma+2n} \]
for nonzero classes $[x_*], [y_*]$ in $H_\bullet (M_*; Q_0)$ if and only if there are elements \[(x_1)_*, \ldots, (x_{n-1})_* \in M_*\] such that 
\begin{align*} (x_*) Sq^2 Sq^1 Sq^2  &= (x_1)_* Sq^1  \\
(x_i)_* Sq^2 Sq^1 Sq^2  &= \begin{cases}
(x_{i+1})_* Sq^1 & i <n-1 \\
y_* & i=n-1.
\end{cases}
\end{align*} 
\end{conj}
The difficulty in using exactly the methods used for $d_2$ to compute higher differentials is that we only computed lifts along $\delta$ for maps that factor through $\FF_2$.  There is no guarantee that such a lift also factors through $\FF_2$ if its composition with $j$ is zero.  

\begin{exmp}\label{exmponesg}
Let $\ell$ denote the generator of $\Upsilon_1$ and $u$ denote $Sq^2 Sq^1 Sq^2 \ell$.
Then in $M_*$, $(u_*)Sq^2 Sq^1 Sq^2 =\ell_*$, so the spectral sequence associated to $\Upsilon_1$ has a differential, 
\[ d_2\left(h_0^{s-\sigma} x_3^{\sigma}[u_*]\right) = h_0^{s-\sigma-1}x_3^{\sigma+2} [\ell_*]  \]
For degree reasons, all higher differentials are zero, so the spectral sequence collapses.  The $E_2$-page of this spectral sequence is depicted in Figure~\ref{figSpSq}.
\begin{figure}[ht!] 
\begin{tikzpicture}[scale=0.7]
\foreach \y in {-1, 1, 2, 3, 4, 5}{
\draw (-.2, \y)--(.2, \y);
\node[left] at (-.2, \y) {\tiny \y};
}
\draw[<->] (-0.5, 0)--(11,0);
\draw[<->] (0, -2.5)--(0,6);
\foreach \x in {0,5}{
\draw[red] (\x, -1.5)--(\x, 5.5);
\foreach \y in {-1.75, -2, -2.25, 5.75, 6, 6.25}{\fill[red] (\x, \y) circle (1pt);}
\foreach \y in { -1, ..., 5}{
\fill[red] (\x,\y) circle (3pt);
}
}
\foreach \x in {4,9}{
\draw[orange] (\x, -1.5)--(\x, 5.5);
\foreach \y in {-1.75, -2, -2.25, 5.75, 6, 6.25}{\fill[orange] (\x, \y) circle (1pt);}
\foreach \y in { -1, ..., 5}{
\fill[orange] (\x-.1,\y-.1) rectangle (\x+.1, \y+.1);
}
}
\foreach \x in {8}{
\draw[black!40!green] (\x, -1.5)--(\x, 5.5);
\foreach \y in {-1.75, -2, -2.25, 5.75, 6, 6.25}{\fill[black!40!green] (\x, \y) circle (1pt);}
\foreach \y in { -1, ..., 5}{
\fill[black!40!green] (\x-.15,\y-.15)--(\x, \y+.15)--(\x+.15, \y-.15);
}
}
\foreach \x in {4,8}{
\foreach \y in {-1, 0, ..., 4} {
\draw[->] (\x+.85, \y+.15)--(\x+.15, \y+.85);}
}
\foreach \x in {1, 2, 3, 6, 7, 10}{
\draw (\x, -.1)--(\x, .1);
\node[below] at (\x, -.1) {\tiny \x};
}
\node at (11.6, 0) {\small t-s};
\node at (-1.5, 3) {\small s};
\node[red] at (.4, .3) {$\ell_*$};
\node[red] at (5.4, .3) {$u_*$};
\node[orange] at (3.3,2) {$x_3^2 \ell_*$};
\node[orange] at (9.8, 2) {$x_3^2 u_*$};
\node[black!40!green] at (7.3, 4) {$x_3^4 \ell_*$};
\end{tikzpicture}
\caption{The $h_0$-localized Davis--Mahowald spectral sequence for $\Upsilon_1$.  The classes marked with a red circle are in degree $\sigma=0$, the classes marked with an orange square are in degree $\sigma=2$, and the classes marked with a green triangle are in degree $\sigma=4$.}
\label{figSpSq}
\end{figure}
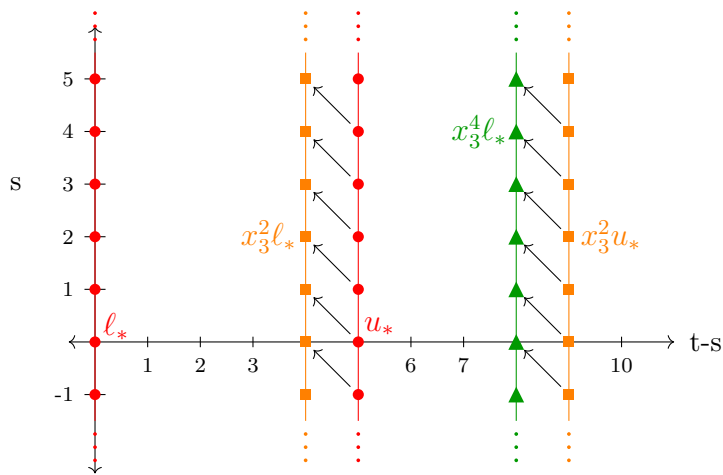
\end{exmp}
Even without a formula for the higher differentials, all differentials can be computed in the spectral sequence associated to a seagull module of any length.  For the infinite seagull module, there is no room for any nonzero differentials.  For a finite seagull, there is a single nonzero differential. 
\begin{prop}\label{propSpSqSeagull}
For any finite $n$, the spectral sequence
\[ \FF_2 [h_0^{\pm 1}, x_3^2] \otimes H_\bullet (\Upsilon_n;Q_0)  \Rightarrow h_0^{-1} \Ext_{\CA(1)}^{\bullet, \bullet}(\Upsilon_n, \FF_2) \]
has precisely one nonzero differential, $d_{2n}$.
\end{prop}
\begin{proof}
It's quick to show via constructing a minimal free resolution of $\Upsilon_n$ that, as an $\FF_2$-vector space, \[h_0^{-1}\Ext_{\CA(1)}^{\bullet, \bullet}(\Upsilon_n, \FF_2) \cong \Ext_{\CA(1)}^{\bullet, \bullet}(\Upsilon_n, \FF_2) \cong \FF_2[h_0^{\pm 1}, x_3^2]/x_3^{2(n+1)}.\]
So, the $E_\infty$-page is comprised of $h_0$-towers in degrees \[t-s=0,4,8, \ldots, 4(n-1).\]
The $Q_0$-homology of $\Upsilon_n$ is $\FF_2\{\ell, u\}$ where $|\ell|=0$ and $|u|=4n+1$.  So, in each even filtration degree, $\sigma=m$, the $E_1$-page has $h_0$-towers in degrees $t-s=2m$ and \\ $t-s=2m+4n+1$. (The $E_1$-page is zero in each odd filtration degree.)  So, the $E_1$-page is comprised of $h_0$-towers in degrees
\[ t-s = 0, 4, 8, \ldots, 4(n-1), \ 4n, \ 4n+1,\  4(n+1),\  4(n+1)+1, \ldots \]
Any differential in this spectral sequence has degree $t-s=-1$.
 So, in order for the towers in degrees $t-s>4(n-1)$ to be eliminated, there must be a differential from each tower in degree $2m+4n+1$ to the tower in degree $2m+4n=2(2n+m)$.  The tower in degree $2m+4n+1$ is in filtration degree $\sigma=m$ and the tower in degree $2(2n+\sigma)$  is in filtration degree $\sigma=2n+m$, so this is a $d_{2n}$ differential.  
\end{proof}

For any bounded below, $Q_0$-local $\CA(1)$-module of finite type, $M$, Theorem~\ref{thmclassification} tells us $M$ is isomorphic to a flock of seagulls. If we can identify that flock of seagulls, we can then compute all differentials in the spectral sequence associated to $M$.  In fact, if $M$ is an arbitrary $\CA(1)$-module, the $Q_0$-
localization of $M$ can be used to compute $h_0^{\pm 1}\Ext_{\CA(1)}^{\bullet, \bullet}(M, \FF_2)$.  

\begin{lemma}\label{lemEinf}
The map $L_0 M \to M$ induces an isomorphism,
\[ h_0^{-1}\Ext_{\CA(1)}^{\bullet, \bullet}(M, \FF_2) \cong h_0^{-1}\Ext_{\CA(1)}^{\bullet, \bullet}(L_0 M, \FF_2) .\]
\end{lemma}
\begin{proof}
The triangle
\[ L_0 M \to M \to L_1 M \]
induces the long exact sequence 
\[\begin{tikzcd}[column sep=tiny] \cdots \ar[r] & 
h_0^{-1}\Ext_{\CA(1)}^{s,t}\left(L_1 M, \FF_2\right) \arrow[r]& h_0^{-1}\Ext_{\CA(1)}^{s,t}\left(M, \FF_2 \right) \arrow[r]\arrow[d,phantom, ""{coordinate, name=Z}]& h_0^{-1}\Ext_{\CA(1)}^{s,t}\left(L_0 M, \FF_2\right)\arrow[dll,rounded corners,to path={ --([xshift=2ex]\tikztostart.east)|- (Z)[near end]\tikztonodes-| ([xshift=-2ex]\tikztotarget.west)-- (\tikztotarget)}]  \\ & h_0^{-1}\Ext_{\CA(1)}^{s+1,t}\left(L_1 M, \FF_2\right)\arrow[r]
& h_0^{-1}\Ext_{\CA(1)}^{s+1,t}\left(M, \FF_2 \right) \arrow[r]
& \cdots  \end{tikzcd}\]
Since $L_1 M$ is $Q_1$-local (equivalently, free as an $\CA(0)$-module) and bounded below, Adams' vanishing theorem \cite[Theorem 2.1]{AdamsPT} implies \[h_0^{-1}\Ext_{\CA(1)}^{\bullet, \bullet}\left(L_1 M, \FF_2 \right)=0. \qedhere\] 
\end{proof}

Consequently, $h_0^{-1}\Ext_{\CA(1)}^{\bullet, \bullet}(M, \FF_2)$ can be computed by identifying the decomposition of $L_0 M = \Upsilon_\infty \otimes M$ into a direct sum of seagulls.  However, computing this decomposition is often difficult.  So, it would be desirable to prove Conjecture~\ref{conjhigherd}, or another way of computing differentials directly from the structure of $M$.

\subsection{\texorpdfstring{Lifting $\CA(1)$-Modules to $\CA$-Modules}{Lifting A(1)-modules to A-modules}
}\label{seclifting}
For any $\CA(n)$, a subalgebra of $\CA$, there is a forgetful functor from $\CA\Mod$ to $\CA(n)\Mod$.
If $M$ is a $\CA(n)$-module in the image of this functor, we say $M$ lifts to an $\CA$-module or that $M$ has a compatible $\CA$-module structure.  

This forgetful functor only exists because there is an action of $\CA$ on $\CA(n)$ that is compatible with the multiplication in $\CA(n)$. (That is, if $x \in \CA(n)$, acting by $x \in \CA$ on $\CA(n)$ is the same as multiplying by $x \in \CA(n)$.) Lin \cite{LinThesis} calls Hopf subalgebras with this property ``nice Hopf subalgebras''.  In \cite{LinThesis}, he shows the only nice Hopf subalgebras of $\CA$ are of the form $\CA(n)$ for some $n$.  So, if $\CB$ is a Hopf subalgebra of $\CA$ not equal to some $\CA(n)$, there is no such forgetful functor and so the question of lifting $\CB$-modules to $\CA$-modules is not well formed.

The following theorem of Davis and its consequences for computing $h_0^{-1}\Ext$ were pointed out to the author by Michael Hopkins. 
\begin{thm}[Davis]\label{thm:Davis}
For any $\CA$-module, $N$, the $h_0$-towers in $\Ext_{\CA}^{\bullet, \bullet}\left(N, \FF_2 \right)$ are in one-to-one correspondence with a basis for $H_\bullet (N;Q_0)$.  
\end{thm}
Davis' proof of this theorem effectively shows
\[ h_0^{- 1}\Ext_{\CA}^{\bullet, \bullet}(N, \FF_2) \cong \FF_2[h_0^{\pm 1}] \otimes H_\bullet (N; Q_0). \]
For any $\CA(1)$-module, $M$, there is an $\CA$-module $\CA \otimes_{\CA(1)} M$ where $\CA$ acts on the left factor. So,
\begin{align*} h_0^{-1}\Ext_{\CA(1)}^{\bullet, \bullet}\left(M, \FF_2 \right) & \cong h_0^{-1}\Ext_{\CA}^{\bullet, \bullet}\left(\CA \otimes_{\CA(1)} M, \FF_2\right) \cong \FF_2[h_0^{\pm 1}] \otimes H_\bullet (\CA \otimes_{\CA(1)} M;Q_0). \end{align*}
If $M$ can be lifted to an $\CA$-module, then there is a shearing isomorphism,
\[ \CA \otimes_{\CA(1)} M \cong \CA\mm\CA(1) \otimes M  \]
where $\CA\mm\CA(1) \otimes M$ has the diagonal action. In this case, 
\begin{align*} h_0^{-1}\Ext_{\CA(1)}^{\bullet, \bullet}\left(M, \FF_2\right) &\cong \FF_2[h_0^{\pm 1}] \otimes H_\bullet \left(\CA\mm\CA(1);Q_0\right) \otimes H_\bullet \left( M ; Q_0\right) \\
&\cong \FF_2[h_0^{\pm 1}] \otimes \FF_2[a] \otimes H_\bullet (M;Q_0) \end{align*}
where the degree of $a$ is $t-s=4$. Hence, the spectral sequence,
\[ \FF_2[h_0^{\pm 1}, x_3^2] \otimes H_\bullet (M;Q_0)   \Rightarrow h_0^{-1}\Ext_{\CA(1)}^{\bullet, \bullet}(M, \FF_2) \]
collapses at the $E_1$-page.  
This collapse depends on $M$ having a compatible $\CA$-module structure.   So, any nonzero differentials, $d_n$ for $n \geq 1$, in the $h_0$-local Davis--Mahowald spectral sequence associated to an $\CA(1)$-module indicate that the module cannot be lifted to an $\CA$-module.  The formula for $d_2$ then implies the following corollary.
\begin{cor}\label{cord2lift}
If $M$, a bounded below $\CA(1)$-module of finite type, has elements $x$ and $y$ such that $[x_*], [y_*] \in H_\bullet (M_*;Q_0)$ are nonzero and $y_* Sq^2 Sq^1 Sq^2  = x_*$ then $M$ does not lift to an $\CA$-module.
\end{cor}
Conjecture~\ref{conjhigherd} would imply the following.
\begin{conj}
Suppose $M$, a bounded below $\CA(1)$-module of finite type, has elements \[x_0, x_1, \ldots, x_{n-1}\] such that $[(x_i)_*] \in H_\bullet(M_*;Q_0)$ is nonzero for all $i$,  $[(x_{n-1})_* Sq^2 Sq^1 Sq^2 ] \in H_\bullet(M_*;Q_0)$ is nonzero, and $(x_i)_* Sq^2 Sq^1 Sq^2  = (x_{i+1})_* Sq^1 $ for all $i<n-1$.  Then $M$ does not lift to an $\CA$-module.  
\end{conj}
While this is only a conjecture in the general case, in some particular cases, we can compute higher differentials and thus detect obstructions to lifting.
\begin{cor}\label{cornolift}
For any finite $n$, $\Upsilon_n$ is not an $\CA$-module. 
\end{cor}
This corollary follows very smoothly from Proposition~\ref{propSpSqSeagull}, though this is certainly not the only way to prove it.  In fact, one can show directly from the Adem relations that $\Upsilon_n$ is an $\CA(m)$- but not an $\CA(m+1)$-module where $m$ is the largest number such that $2^{m-1}$ divides $n$.
\begin{cor}\label{corliftingcrit}
Let $M$ be a bounded below $\CA(1)$-module of finite type.
\begin{enumerate}[label=(\roman*)]
    \item If $L_0 M$ is stably equivalent to a flock of seagulls including a finite seagull, then $M$ does not lift to an $\CA$-module.
    \item If $M$ is $Q_0$-local, then $M$ lifts to an $\CA$-module if and only if $M$ is stably equivalent to a flock of infinite seagulls (or zero). 
\end{enumerate}\end{cor}
\begin{proof}
As shown in Lemma~\ref{lemEinf}, $h_0^{-1}\Ext^{s,t}_{\CA(1)}(M, \FF_2)\cong h_0^{-1}\Ext^{s,t}_{\CA(1)}(L_0 M, \FF_2)$. Furthermore, $L_0 M \to M$ induces an isomorphism in $Q_0$-homology.  
So, there is at least one nonzero differential in the spectral sequence
\[\FF_2[h_0^{\pm 1}, x_3^2] \otimes H_\bullet(M;Q_0)  \Rightarrow h_0^{-1}\Ext_{\CA(1)}^{s,t}(M, \FF_2)\]
if and only if there is at least one nonzero differential in the spectral sequence
\[\FF_2[h_0^{\pm 1}, x_3^2
] \otimes H_\bullet(L_0 M;Q_0) \Rightarrow h_0^{-1}\Ext_{\CA(1)}^{s,t}(L_0 M, \FF_2).\]
Then, $(i)$ follows from the fact that if $L_0 M$ is stably equivalent to a flock of seagulls including a finite seagull, the spectral sequence associated to $L_0 M$ has a nonzero differential.

Part $(ii)$ follows quickly from the classification theorem and Corollary~\ref{cornolift}.
\end{proof}

When $M$ is finite, part $(ii)$ follows directly from the classification theorem and a result of Palmieri \cite[A.1]{PalmieriNil2} that shows any finite, $Q_0$-local $\CA$-module has zero $Q_0$-homology. 

\bibliographystyle{plain}	
\bibliography{refs}		


\end{document}